\documentclass[12pt,twoside]{amsart}
\usepackage{amsmath}
\usepackage{epsfig}
\usepackage{amssymb}
\usepackage{xcolor}
\usepackage{blkarray}

\usepackage{version}

\excludeversion{percent}

\usepackage{hyperref}

\newtheorem{theorem}{Theorem}[section]
\newtheorem{lemma}[theorem]{Lemma}
\newtheorem{corollary}[theorem]{Corollary}
\newtheorem{proposition}[theorem]{Proposition}

\newtheorem{sublemma}{}[theorem]

\newtheorem{conjecture}[theorem]{Conjecture}

\theoremstyle{definition}

\theoremstyle{remark}

\numberwithin{equation}{section}

\newcommand{\ba}{\backslash}

\newcommand{\si}{{\rm si}}
\newcommand{\cl}{{\rm cl}}

\newcommand{\subproof}{\begin{proof}[Subproof]}

\newcommand{\lm}{\lambda}

\newcommand{\cT}{\mathcal{T}}

\newcommand{\cS}{\mathcal{S}}
\newcommand{\cW}{\mathcal{W}}

\DeclareMathOperator{\breadth}{breadth}
\DeclareMathOperator{\guts}{guts}
\DeclareMathOperator{\coguts}{coguts}
\DeclareMathOperator{\inter}{int}

\begin{document}

\sloppy

\title{What is a $4$-connected matroid?}

\author[N.\ Brettell]{Nick Brettell}
\address{School of Mathematics and Statistics, Victoria University of Wellington, New Zealand}
\email{nick.brettell@vuw.ac.nz}
\author[S.\ Jowett]{Susan Jowett}
\address{School of Mathematics and Statistics, Victoria University of Wellington, New Zealand}
\email{swoppit@gmail.com}
\author[J.\ Oxley]{James Oxley}
\address{Department of Mathematics, Louisiana State University, Baton Rouge, Louisiana, USA}
\email{oxley@math.lsu.edu}
\author[C.\ Semple]{Charles Semple}
\address{School of Mathematics and Statistics, University of Canterbury, New Zealand}
\email{charles.semple@canterbury.ac.nz}
\author[G.\ Whittle]{Geoff Whittle}
\address{School of Mathematics and Statistics, Victoria University of Wellington, New Zealand}
\email{geoffwhittle201@gmail.com}

\subjclass[2010]{05B35}
\date{\today}
\thanks{The first, fourth, and fifth authors were supported by the New Zealand Marsden Fund.}

\begin{abstract}
The {\em breadth} of a tangle $\cT$ in a matroid is the size of the largest spanning uniform
submatroid of the tangle matroid of $\cT$. A matroid $M$
is {\em weakly $4$-connected}
if it is 3-connected and whenever $(X,Y)$ is a partition of $E(M)$ with $|X|,|Y|>4$, then $\lm(X)\geq 3$.
We prove that if $\cT$ is a tangle of order $k\geq 4$ and breadth $l$
in a matroid $M$, then $M$ has a 
weakly 4-connected minor $N$ with a tangle $\cT_N$ of order $k$, breadth $l$ and has the property
that $\cT$ is the tangle in $M$ induced by $\cT_N$. A set $Z$ of elements of a matroid $M$ is $4$-{\em connected} if 
$\lm(A)\geq\min\{|A\cap Z|,|Z-A|,3\}$ for all $A\subseteq E(M)$. As a corollary of our theorems
on tangles 
we prove that if $M$ contains an $n$-element $4$-connected set where $n\geq 7$, then $M$ 
has a weakly
$4$-connected minor that contains an $n$-element $4$-connected set. \end{abstract}

\maketitle

\section{Introduction} 

This introduction assumes some familiarity with matroid tangles. Definitions and basic properties
are given in Section~\ref{tangles}.

Tangles were introduced by Robertson and Seymour \cite{RS91} to capture highly connected regions of 
a graph and they noted \cite[p.190]{RS91} that tangles extend  to matroids. 
A  tangle of order $k$
in a matroid can be thought of as capturing a ``$k$-connected region'' of the matroid. 
A matroid with such a tangle may bear little relation to a matroid that is ``$k$-connected''
in some more concrete sense, but it is natural to expect that a tangle of order $k$ guarantees a minor
that possesses a more concrete connectivity property. 

For $k\in\{2,3\}$ the relationship between $k$-tangles and existing notions of 
connectivity is clear. A tangle of order 2 in a matroid corresponds to a 
connected component of the matroid with at least two elements. A tangle of order 3 corresponds to a 
3-connected part of the canonical 2-sum decomposition of the matroid.
We discuss the precise connection in Section~\ref{tangles}, but the point is that tangles of order
2 and 3 identify 2- and 3-connected minors that correspond to the 2- or 3-connected ``region'' captured by
 the tangle.
 
This paper considers the case $k=4$. If a tangle of order 4 identifies a 4-connected ``region'' of a matroid,
then one would expect to find a minor associated with the tangle that satisfies a more tangible form of 
4-connectivity. 
Various notions of 4-connectivity have been considered in the literature.
It turns out that ``weak 4-connectivity'' is the 
connectivity that we can guarantee in a minor of a matroid with a  4-tangle.
A matroid $M$
is {\em weakly $4$-connected}
if it is 3-connected and whenever $(X,Y)$ is a partition of $E(M)$ with $|X|,|Y|>4$, then $\lm(X)\geq 3$. 
We prove that a matroid $M$ with a tangle $\cT$ of order $k\geq 4$ has a
weakly $4$-connected minor $N$ with a tangle $\cT_N$ of order $k$.

All well and good, but we need more for a satisfactory answer. We need guarantees 
that the information in $\cT$ is not significantly eroded in $\cT_N$.
To obtain that guarantee we would like to have the property that the ``size'' of $\cT$ is  
preserved in $\cT_N$. It is also desirable 
that $\cT$ and $\cT_N$ are related in a meaningful way. 

To deal with the issue of size we define the {\em breadth}
of a tangle $\cT$ to be the number of elements in a largest spanning uniform submatroid of
the tangle matroid $M_\cT$. This generalises cardinality in the sense that, if $M$ is Tutte 
$k$-connected, then the breadth of its unique tangle of order $k$ is $|E(M)|$. 

We relate the structure of $\cT_N$ to that of $\cT$ as follows.
Recall that a tangle
is a collection $\cT$ of subsets of $E(M)$ that act as pointers to our $k$-connected region.
The $k$-tangle $\cT$ {\em generates} a tangle $\cT_N$ in the minor $N$ if $\cT_N$ is the 
unique tangle of order $k$ in $N$ that contains the collection of intersections of the
members of $\cT$ with $E(N)$. 

With these two notions in hand we can state our main theorem.

\begin{theorem}
\label{biggy}
Let $\cT$ be a tangle of order $k\geq 4$ in a matroid $M$. Then $M$ has a weakly
$4$-connected minor $N$ with a tangle $\cT'$ of order $k$ such that 
$\cT$ generates $\cT'$ in $N$ and such that the breadth of $\cT'$ is equal to that of $\cT$.
\end{theorem}

The results of this paper also have a connection with $k$-connected sets.
Let $M$ be a matroid and $k\geq 2$ be an integer. A subset $Z$ of $E(M)$ is
$k$-{\em connected} if $\lm(A)\geq\min\{|A\cap Z|,|Z-A|,k-1\}$ for all $A\subseteq E$.  
We show that if $Z$ is a $k$-connected set of size at least $3k-5$, then there is a tangle $\cT$ of order $k$
in $M$ such that $M_\cT|Z\cong U_{k-1,|Z|}$. We  note that the relationship between $k$-connected sets and uniform submatroids of tangle matroids is further motivation for our definition of breadth.
Via this connection we obtain the following theorem.

\begin{theorem}
\label{get-weak}
Let $Z$ be an $n$-element $k$-connected set in the matroid $M$ where $n\geq 3k-5$ and $k\geq 4$. 
Then $M$ has a weakly $4$-connected minor $N$ that contains an $n$-element 
$k$-connected set.
\end{theorem}

We now discuss the reasons for our interest in this problem.
The unavoidable minors of large 3-connected matroids are studied in \cite{doov}. 
It is natural to seek analogous results for 4-connected matroids and we are currently engaged
in a project with that as a goal. But while there is general agreement as to what a 
3-connected matroid is, 4-connectivity is somewhat more vexed. Tutte 4-connectivity
is a stringent condition that   fails, for example, for all projective geometries of rank at least three. 
In practise,  various weaker notions
of 4-connectivity have been considered; these include vertical 4-connectivity, cyclic 4-connectivity,
internal 4-connectivity, sequential 4-connectivity and weak 4-connectivity. Which connectivity
should we begin with in our search for unavoidable minors? It is surely a truism in
mathematics that one should select the weakest hypotheses for which one expects the theorem
to hold. With that in mind, our weakest beginning in our search for unavoidable minors
is to start with a matroid with a ``large'' tangle
of order 4 and our initial goal is to find the strongest version of 4-connectivity that such a tangle 
guarantees in a minor. 

The results of this paper imply that the unavoidable minors for a large weakly 4-connected matroid 
are also the unavoidable minors for a matroid with a ``large''  tangle of order 4. We lose no generality 
in focussing on the weakly 4-connected case.

The paper is structured as follows. Section~2 deals with technical preliminaries. Section~3 introduces
tangles and the tangle matroid. We give the formal definition of breadth and
discuss the relationship between $k$-connected sets and breadth.
Section~4 considers tangles generated in minors. In Section~5 we find sufficient conditions to be 
able to move to a proper minor without damaging the breadth of a given tangle. Section~6 considers 
the structure of flats in tangle matroids of low rank. Finally, in Section~7 we are able to prove that,
given a tangle of order at least 4 in a matroid $M$, we can move to a weakly 4-connected minor
with an associated tangle whose breadth is equal to that of the original tangle.
Section~8 discusses the special case of tangles of order exactly 4. In Section~9 we give an
example to show that our main results are, in a sense, best possible. In Section~10 we give
the short proof of Theorem~\ref{get-weak}. Finally
we consider some open problems and conjectures in Section~11.

We were in the final stages of writing this paper when we became aware of a recent paper
of Carmesin and Kurkofka \cite{carm}. They study essentially the same problem for 
4-tangles in graphs as we do for matroids. 
Their outcome is to find an internally 4-connected minor. This is a stronger property than
weak 4-connectivity. Examples are given in Section~\ref{better} that show that 
we cannot improve on weak 4-connectivity, even for the highly structured class of graphic matroids. 
This is due to our requirement of 
preserving breadth.

\section{Preliminaries}


We follow Oxley~\cite{O11} for any unexplained matroid terminology or notation. 
Note that when we refer to a {\it partition} of a set, we do not require that each subset 
in the partition is nonempty. 
For a matroid $M$,  the {\em connectivity function} $\lm_M$   is   
defined, for all subsets $A$ of $E(M)$, by $\lm_M(A)=r_M(A)+r_M(E(M)-A)-r(M)$. 
We say that the set $A$ and the partition $(A,E(M) - A)$ are $k$-{\em separating} if $\lm(A)<k$; they are {\em exactly $k$-separating}
if $\lm(A)=k-1$. 

The {\em coclosure operator} of $M$, denoted $\cl^*_M$ or just $\cl^*$, is defined, for all subsets $A$ of $E(M)$, 
by $\cl^*(A)=\cl_{M^*}(A)$. A set $A$ is {\em coclosed} if $\cl^*(A)=A$. 
It is easily seen that $x\in\cl^*(A)-A$ if and only if 
$x$ is a coloop of $M|(E(M)-A)$, in particular $A$ is coclosed if and only if $M(E(M)-A)$ has no coloops.
The next result is well known (see, for example, \cite[Proposition 2.1.12]{O11}). 
When we say {\it by orthogonality}, we shall mean by an application of this lemma.

\begin{lemma}
\label{coclos}
Let $M$ be a matroid and $(A,\{x\},B)$ a partition of $E(M)$. Then 
$x\in\cl^*(A)$ if and only if $x\notin \cl(B)$.
\end{lemma}

 A set in a matroid is {\em fully closed} if it is both closed and coclosed.
Fully closed sets play an important role in this paper. We make frequent use of the next elementary fact.

\begin{lemma}
\label{keep-fcl}
Let $A$ be a fully closed set in a matroid $M$ and let $N$ be a minor of $M$ whose ground set
contains $E(M)-A$. Then $A\cap E(N)$ is fully closed in $N$. 
\end{lemma}

\begin{proof}
Say that $N=M/x$. By assumption $x\in A$.  By definition $\cl_{M/x}(A-\{x\})=\cl_M(A)\cap E(N)=A-\{x\}$,
so that $A-\{x\}$ is closed in $N$. Also $\cl^*_{M/x}(A-\{x\})=\cl^*_M(A-\{x\})\cap E(N)$.
Since $A$ is coclosed in $M$, we have $\cl^*_M(A-\{x\})\subseteq A$, and it follows that
$A-\{x\}$ is coclosed in $M/x$.  The proof now follows by duality and a routine induction.
\end{proof}

We make free use in proofs of the next result (see, for example,  \cite[Proposition 8.2.14]{O11}).

\begin{lemma}
\label{up-down}
Let $A$ be a set of elements in a matroid $M$. Then the following 
hold for each $x$ in $E(M) - A$.
\begin{itemize}
\item[(i)] $\lm(A\cup\{x\})=\lm(A)-1$ if and only if $x\in\cl(A)$ and $x\in\cl^*(A)$.
\item[(ii)] $\lm(A\cup\{x\})=\lm(A)$ if and only if $x$ belongs to exactly one of
$\cl(A)$ and $\cl^*(A)$.
\item[(iii)] $\lm(A\cup\{x\})=\lm(A)+1$ if and only if $x\notin \cl(A)$ and $x\notin\cl^*(A)$.
\end{itemize}
\end{lemma}

We also freely use the next elementary lemma.

\begin{lemma}
\label{tikokino}
Let $M$ be a matroid, $A\subseteq E(M)$, and $x$ be an element of $E(M)-A$ that is not a loop of $M$. Then 
$\lm_{M/x}(A)=\lm_M(A)-1$ if $x\in\cl_M(A)$, and otherwise $\lm_{M/x}(A)=\lm_M(A)$.
\end{lemma}
 
For a set $X$ in a matroid $M$, the {\em guts of $X$}, denoted $\guts(X)$,
is the set $\cl(X)\cap \cl(E(M) - X)$; the {\em coguts of $X$}, denoted $\coguts(X)$, 
is the set $\cl^*(X)\cap \cl^*(E(M) - X)$. If $X$ is fully closed, then $\guts(X)=\cl(E(M) - X)\cap X$
and $\coguts(X)=\cl^*(E(M) - X)\cap X$. Clearly, both these sets are contained in 
$X$. The set $X-(\guts(X)\cup \coguts(X))$ is the {\em interior of $X$}, denoted 
$\inter(X)$.

\begin{lemma}
\label{2int}
Let $X$ be a set in the matroid $M$. If $\inter(X)$ is nonempty, then $|\inter(X)|\geq 2$.
\end{lemma}

\begin{proof}
Assume that $\inter(X)\neq \emptyset$. Say $e\in\inter(X)$. Let $Y=E(M)-X$. Since $e\notin\cl(Y)$, there is a
cocircuit $C^*$ of $M$ with $e\in C^*\subseteq X-\cl(Y)$. By duality there is a circuit $C$ of $M$
with $e\in C\subseteq X-\cl^*(Y)$. Now there is an element $f\neq e$ such that $f\in C\cap C^*$. The lemma follows as 
$C\cap C^*\subseteq \inter(X)$.
\end{proof}

\begin{lemma}
\label{disjoint}
Let $F$ be a $3$-separating set in a $3$-connected matroid $M$ where $F$
is fully closed and has at least three elements. 
Then $\guts(F)\cap\coguts(F)=\emptyset$.
\end{lemma}

\begin{proof}
Say $x\in \guts(F)\cap\coguts(F)$. Then $|E(M)| \ge 4$. It is readily checked that a
 $3$-connected matroid $M$ with at least four elements is both simple and cosimple. 
Hence $|E(M) - F| \geq 3$. 
By Lemma~\ref{up-down}(i), $\lm((E(M) - F)\cup\{x\})<2$, that is, $\lm(F-\{x\})<2$.
But $|F-\{x\}|\geq 2$ and we have contradicted the assumption that $M$ is
3-connected.
\end{proof}

\begin{lemma}
\label{guts-coguts}
Let $(F,G)$ be a $3$-separating partition in a $3$-connected matroid $M$ where $F$ is fully 
closed  and has at least three elements. 
If the guts of $F$ and the coguts of $F$ are both nonempty, then
$|\guts(F)|=|\coguts(F)|=1$.
\end{lemma}

\begin{proof}
Assume that $|\coguts(F)|\geq 2$; say $x,y\in\coguts(F)$. Since $F$ is fully closed,
$x,y\in F$. Now $x,y\in\coguts(F)$ so that $x$ and $y$ are coloops of $M|F$ and hence
$r_{M\ba\{x,y\}}(F-\{x,y\})=r_M(F)-2$. If $G=\emptyset$, then $x$ and $y$ are coloops of $M$
contradicting the fact that $M$ is 3-connected. Hence $|E(M)|\geq 4$. Now, if $r(M\ba\{x,y\})<r(M)$ we 
deduce that $(\{x,y\},E(M)-\{x,y\})$ is a 2-separation of $M$ contradicting the assumption that
$M$ is 3-connected. Hence $r(M\ba\{x,y\})=r(M)$. Thus  $\lm_{M\ba\{x,y\}}(F-\{x,y\}) = 0$. 

Take $z$ in $\guts(F)$. Then,  
by Lemma~\ref{disjoint}, $z\notin\{x,y\}$. Hence $z\in F-\{x,y\}$ and 
$ z \in \cl_{M\ba x,y}(G)$. Thus $G$ and $G \cup \{z\}$ are $1$-separating in $M\ba \{x,y\}$, 
so $z$ is a coloop of $M\ba \{x,y\}$, contradicting the fact that $z \in \cl(G)$.
\end{proof}

\begin{lemma}
\label{interior-decoration}
Let $(F,G)$ be a $3$-separating partition  in a $3$-connected matroid $M$ where $F$ is fully
closed and $|F|\geq 3$. Then one of the following holds.
\begin{itemize}
\item[(i)] $M$ has $F$ as a line, $\guts(F)=F$, and $\coguts(F)=\inter(F)=\emptyset$.
\item[(ii)]  $M^*$ has $F$ as a line, $\coguts(F)=F$,
and $\guts(F)=\inter(F)=\emptyset.$
\item[(iii)] $F$ is a $4$-element fan, $|\guts(F)|=|\coguts(F)|=1$, and $|\inter(F)|=2$.
\item[(iv)] $|\inter(F)|\geq 3$.
\end{itemize}
\end{lemma}

\begin{proof}
Since $(F,G)$ is a $3$-separating partition of $M$, 
$$r(\guts(F)) = r(F \cap \cl(G)) \le r(F) + r(G) - r(M) \le 2.$$
Assume that $\coguts(F)=\emptyset$.  Let $F'=F-\guts(F)$. If $F'=\emptyset$,
then (i) holds. 
Next assume that $|F'| \in \{1,2\}$. Then, as $F' \not \subseteq \cl(G)$, we see that 
$F' \not \subseteq \cl(E(M) - F')$. Hence $r(E(M) - F') \le r(M) - 1$. Thus the 3-connected matroid $M$, which has at least six elements, 
has a cocircuit with at most two elements, a contradiction.

We now know that if $\coguts(F) = \emptyset$, then   (i) or (iv) holds. Dually,
if  $\guts(F) = \emptyset$, then   (ii) or (iv) holds. By Lemma~\ref{guts-coguts},
the remaining case is
when $|\guts(F)|=|\coguts(F)|=1$. In this case, if $|F|>4$, then (iv) holds; 
if $|F|=4$, then $|\inter(F)|=2$, so $\inter(F) \cup \guts(F)$ is a triangle, while $\inter(F) \cup \coguts(F)$ is a triad. Thus (iii) holds. 
\end{proof}

The following well-known consequence of the submodularity of the connectivity function will be useful. 

\begin{lemma}
\label{set-diff}
Let $M$ be a matroid and let $A$ and $B$ be subsets of $E(M)$. Then
$$\lm(A)+\lm(B)\geq \lm(A-B)+\lm(B-A).$$
\end{lemma}

\begin{proof}
Let $A'=E(M)-A$ and $B'=E(M)-B$. Then $A-B=A\cap B'$ and $B-A=A'\cap B$. 
By symmetry, $\lm(A'\cup B)=\lm(A\cap B')$. We have
\begin{align*}
\lm(A)+\lm(B)&=\lm(A')+\lm(B)\\
&\geq \lm(A'\cap B)+\lm(A'\cup B)\\
&=\lm(A'\cap B)+\lm(A\cap B')\\
&=\lm(B-A)+\lm(A-B).
\end{align*}
\end{proof}

\section{Tangles}
\label{tangles}

Let $M$ be a matroid and $k$ be a positive integer. A {\em tangle of 
order $k$ in $M$} is a collection $\mathcal T$ of subsets of $E(M)$ such that the following
hold.
\begin{itemize}
\item[(T1)] If $A\in \cT$, then $\lm_M(A)<k-1$.
\item[(T2)] If $A \subseteq E(M)$ and $\lm_M(A)<k-1$,  then $A$ or $E(M) - A$ is in $\cT$.
\item[(T3)] If $A,B,C\in\cT$, then $A\cup B\cup C\neq E(M)$.
\item[(T4)] If $e\in E(M)$, then $E(M)-\{e\}\notin \cT$.
\end{itemize}

It is proved in \cite[Lemma 3.1]{GGRW06} that, to verify that $\cT$ is a tangle, we may replace (T3) by the following
pair of conditions.
\begin{itemize}
\item[(T3a)] For $B\in \cT$ and  $A\subseteq B$, if $\lm_M(A)<k-1$, then $A\in \cT$.
\item[(T3b)] If $(A_1,A_2,A_3)$ is a partition of $E(M)$, then $\cT$ does not contain all three
of $A_1$, $A_2$ and $A_3$.
\end{itemize}

Note that our definition of the order of a tangle accords with that used in 
\cite{GGW09, GGRW06} but differs from that used 
in \cite{GZ15, H15}, where  what we have called a tangle of order $k$ is called a tangle of order $k - 1$. 
If $\cT$ is a tangle of order $k$ in $M$, then we say that a $(k-1)$-separating subset $A$ of $E(M)$
is $\cT$-{\em small} if $A\in \cT$; otherwise $A$ is $\cT$-{\em large}. A subset $W$ of $E(M)$
is $\cT$-{\em weak} if it is contained in a $\cT$-small set; otherwise it is $\cT$-{\em strong}. 
We also say that $\cT$ is a $k$-{\em tangle}.

The following observation may aid the reader's intuition for tangles. Say $\cT$ is a tangle of order
$k$ in a matroid $M$. The way that $\cT$ identifies a $k$-connected region of $M$ is as follows.
If $(A,B)$ is a partition of $E(M)$ with $\lm(A)\leq k-2$, then exactly one of $A$ or $B$ belongs to 
$\cT$. If $A\in \cT$, then we are saying that the region is primarily on the $B$ side and can think of $A$
as pointing towards that region. In other words we may define tangles by {\em orienting} such partitions.
We at times use the language of orientations as follows. If we are defining a $k$-tangle, then for each
set $A$ with $\lm(A)\leq k-2$ we regard the choice of  $A$ belonging to $\cT$ or otherwise as a 
choice of {\em orientation} of $A$. 

Tangles were introduced by Robertson and Seymour \cite{RS91}  and they noted \cite[p.190]{RS91} 
that tangles extend  to matroids. This was later done \cite{D96, GGW09, GGRW06}. 

In view of the following result~\cite[p.11]{GZ15}, the rest of the paper will focus on tangles of order at least two.

\begin{lemma}
\label{nulltangle}
Let $M$ be a nonempty matroid. Then the  empty set is the unique tangle of order $1$ on $M$.
\end{lemma}

A tangle $\cT$ of order 2 corresponds to a connected 
component of $M$ with at least two elements. To see this, say that $X$ is such a component.
For a tangle of order 2, we only have to consider subsets $A$ of $E(M)$
with $\lm(A)=0$. For such a set $A$, either $X\subseteq A$ or $A\cap X=\emptyset$.
Define $\cT$ by choosing all such sets $A$ with  $X \cap A = \emptyset$. It is readily
verified that $\cT$ is a tangle. Moreover it is straightforward to verify that all tangles of order 2 in $M$
arise from a component of $M$ in the above way.

It is a little more subtle to show that the tangles of order 3 in a matroid are in bijective correspondence with
the 3-connected parts of  
the canonical 2-sum decomposition of $M$. Here something is lost. Each 3-connected
part of the  2-sum decomposition is determined up to isomorphism, 
but these parts are minors of the
original matroid and the ground sets of these minors are never uniquely determined. 
Apart from that quibble, there is a perfectly satisfactory relationship between the 3-tangles
of a matroid and the parts of the canonical 2-sum decomposition. The proof of these observations is not
difficult, but full details require more space than we have available. The following observation may
help. Say that $N$ is a minor of $M$ corresponding to a 3-connected part of the 2-sum decomposition
of $M$ with at least 4 elements. For a tangle $\cT$ of order 3 we need to orient sets $A$ with $\lm(A)\leq 1$.
Say that $A$ is such a set. Then either $A$ or $E(M)-A$ contains at most one element of $N$.
We choose $A\in\cT$ if $A$ contains at most one element of $N$. Then $\cT$ is a tangle of order 3. 
Again, it can be verified that every tangle of order 3 in $M$ arises in the above way.

As noted in the introduction, the problem is immediately more vexed for 4-connectivity. 
The  various weaker notions
of 4-connectivity that have been considered, for example vertical 4-connectivity, cyclic 4-connectivity,
internal 4-connectivity, and sequential 4-connectivity, all have the property
that, apart from degenerately
small examples, a matroid with any of the above types of 4-connectivity will have a unique tangle of
order 4. What about the converse? Given a tangle of order 4 in a matroid $M$, 
is there a version of 4-connectivity such that $M$ is guaranteed to have a minor 
$N$ with this type of 4-connectivity? This poorly posed   question clearly needs refinement. 
In fact, we want the minor $N$ to have more than just  a connectivity property; we want
the information in the minor to relate to that of the tangle in a significant way. Returning to the
trivial example of 2-tangles, we want to identify the particular  component associated with the tangle, not just any component. 
Furthermore, we want to guarantee that information in the tangle has not been lost. Specifically, we want to know precisely the component captured by the 2-tangle  rather than just some 
proper minor of this component.

One issue that occurs with tangles is   measuring their ``size''. It is natural to view the ``size'' of a
$2$-tangle in a matroid $M$ as the cardinality of the ground set of the component of $M$ that it captures,
and an analogous comment clearly applies to $3$-tangles. For tangles of higher order, things
become a little more complicated.

Let $\cT$ be a tangle of order $k$ in the matroid $M$. A $\cT$-small subset is {\em maximal}
if it is not properly contained in any other $\cT$-small subset.  We define the {\em cover-size} of $\cT$ to be the minimum number of 
$\cT$-small sets whose union is $E(M)$. Cover size is generally not well
defined for tangles of order 2, but, for tangles of order at least $k\geq 3$, each singleton $x$
has $\lm(\{x\})\leq k-2$. By (T4) $\{x\}$ must be $\cT$-small.
Thus, in this case,  the ground set of $M$ can be covered by $\cT$-small sets.
For tangles of order 3,
the cover-size of $\cT$ is equal to the size of the ground set of a member of the isomorphism class
of 3-connected minors that it captures.  While cover size is a natural measure, we were led
to an alternative measure that relates well both 
to ground-set cardinality
for highly connected matroids and to $k$-connected sets.
That version relies on the tangle matroid and we turn to this topic now.

\subsection*{Tangle matroids and breadth}

Tangle matroids were introduced in \cite{GGRW06}. The next theorem is \cite[Lemma 4.3]{GGRW06}.

\begin{theorem}
\label{tangle-matroid}
Let $\cT$ be a tangle of order $k$ in a   matroid $M$ and let $\mathcal H$ be the 
collection of maximal $\cT$-small sets. Then $\mathcal H$ is the collection of hyperplanes
of a rank-$(k-1)$ matroid on $E(M)$.
\end{theorem}

The matroid defined by Theorem~\ref{tangle-matroid} is the {\em tangle matroid of $\cT$} and is
denoted $M_\cT$. Hall~\cite{H15} proved the following characterisation of tangle matroids. 

\begin{theorem}
\label{dennis}
A matroid $M$ other than $U_{1,1}$ is a tangle matroid if and only if $M$ has no three hyperplanes whose union is $E(M)$.
\end{theorem}

The next lemma summarises some basic properties of the tangle matroid.
Note that, since tangles depend only on the connectivity function of $M$, the collection
$\cT$ is a tangle in $M$
if and only if $\cT$ is a tangle in $M^*$.

\begin{lemma}
\label{tangle-matroid-properties}
Let $\cT$ be a tangle of order $k$ in a matroid $M$. Then the following hold for $M_\cT$ and all subsets 
$A$ of $E(M)$.
\begin{itemize}
\item[(i)] If $A$ is 
$\cT$-strong, then $r_{M_\cT}(A)=k-1$;  
otherwise, $r_{M_\cT}(A)=\min\{\lm_M(B):B\supseteq A \text{~and  $B$ is $\cT$-small}\}$.
\item[(ii)] $A$ is a basis of $M_\cT$ if and only if $A$ is $\cT$-strong
and $|A|=k-1$.
\item[(iii)] If $|A|<k-1$, then $A$ is independent in $M_\cT$ if and only if $A$ is $\cT$-small
and there is no $\cT$-small set $B$ containing $A$ with $\lm(B)<|A|$.
\end{itemize}
\end{lemma}

We use the tangle matroid to obtain our alternative measure of size. Let $\cT$ be a tangle of
order $k$ in a matroid $M$. Then the {\em breadth} of $\cT$ is the cardinality of a largest spanning
uniform restriction of $M_\cT$.

Consider a trivial example. In a matroid $M$, let $C$ be a connected component with at least two elements, and let 
$\cT$ be the tangle of order 2 defined as follows. A set $Z$ with $\lm(Z)=0$ is $\cT$-large if
$C\subseteq Z$;  otherwise it is $\cT$-small. Evidently, $r(M_\cT)=1$. Moreover,  
$M_\cT|C\cong U_{1,|C|}$ and all   elements of $E(M_\cT) - C$ are loops. 
Thus the  breadth of $\cT$ is $|C|$. 

For a slightly less trivial example, let $\cT$ be a tangle of order 3 in $M$. We may assume that
$M$ is connected as the presence of other components simply adds loops to the tangle matroid.
We have $r(M_\cT)=2$, and the breadth of $\cT$ is the maximum size of a rank-2 uniform 
restriction of $M_\cT$. This is just the number of parallel classes in $M_\cT$. 
Note that each parallel class identifies a maximal $\cT$-small set. Each such $\cT$-small
set contains an element of the 3-connected minor $N$ of $M$ that $\cT$ captures. Hence the
breadth of $\cT$ is equal to the cardinality of this minor. Note that a 3-connected matroid $M$ 
with at least four elements
has a unique tangle $\cT$ of order 3 and the breadth of $\cT$ is equal to $|E(M)|$. This correspondence 
works more generally. 

\begin{lemma}
\label{k-conn}
Let $M$ be a $k$-connected matroid other than $U_{1,1}$ where $k\geq 2$ and $|E(M)|>3(k-2)$. 
Then $M$ has a unique tangle $\cT$
of order $k$. In particular, $\cT = \{A \subseteq E(M): |A| \le k-2\}$.
Moreover, $M_\cT\cong U_{k-1,|E(M)|}$ and the breadth of $\cT$ is equal to 
$|E(M)|$.
\end{lemma}

\begin{proof}
Say that $\cT$ is a tangle of order $k$ in $M$.
We are required to orient sets $A$ with $\lm(A)\leq k-2$. For such sets we have either $|A|\leq k-2$
or $|E(M)-A|\leq k-2$. Assume that $A\in\cT$ where $|E(M)-A|\leq k-2$. Let $A'$ be a maximal
member of $\cT$ containing $A$. Say that $|E(M)-A'|\leq 1$. By the definition of tangle,
singletons are $\cT$-small. Hence we easily obtain a cover of $E(M)$ with three $\cT$-small sets.
Say $|E(M)-A'|\geq 2$. Let $(X,Y)$ be a partition of $E(M)-A'$ into nonempty subsets. Since $A'$ is
maximal, $X$ and $Y$ are $\cT$-small. Hence $\{A,X,Y\}$ is a cover of $E(M)$ by $\cT$-small sets,
again contradicting the definition of tangle. 

It follows from the above that the only way to choose members of $\cT$ is to choose all sets 
$A$ with $|A|\leq k-2$. It is routinely verified that, under the hypotheses of the lemma, this choice leads to a well-defined tangle.
\end{proof}

\subsection*{Breadth and $k$-connected sets}
We defined the breadth of a tangle using uniform submatroids of the tangle matroid. Such submatroids
also give rise to $k$-connected sets.

\begin{lemma}
\label{get-k-con}
Let $\cT$ be a tangle of order $k$ in the matroid $M$ and assume that $Z\subseteq E(M)$ has the
property that $|Z|\geq k-1$ and that  $M_\cT|Z\cong U_{k-1,|Z|}$. Then $Z$ is a $k$-connected set in $M$.
\end{lemma}

\begin{proof}
Assume that $Z$ satisfies the hypotheses of the lemma, but that $Z$ is not a $k$-connected set.
Then, up to symmetry, there is a partition $(A,B)$ of $E(M)$ such that either (i)
$|A\cap Z|<k-1,|B\cap Z|\geq |A\cap Z|$ and $\lm(A)<|A\cap Z|$, or
(ii) $|A\cap Z|,|B\cap Z|\geq k-1$, and $\lm(A)<k-1$. Consider (i). If $A$ is $\cT$-small, then 
$r_{M_\cT}(A\cap Z)\leq \lm(A)<|A\cap Z|<k-1$. This implies that $A\cap Z$ contains a circuit 
of $M_\cT$ of size at most $k-2$, contradicting the fact that $M_\cT|Z\cong U_{k-1,|Z|}$.
We obtain the same contradiction in the case that $B$ is $\cT$-small. For (ii) we may assume
that $A$ is $\cT$-small. In this case we obtain the contradiction that $A\cap Z$ contains a 
circuit of $M_\cT$ of size at most $k-1$. Hence $Z$ is a $k$-connected set of $M$.
\end{proof}

On the other hand sufficiently large   $k$-connected sets give rise to tangles.

\begin{lemma}
\label{get-k-tangle}
Let $k\geq 3$ be an integer and let $Z$ be a $k$-connected set in the matroid $M$ such that
$|Z|\geq 3k-5$. Let $\cT_Z$ be the collection of subsets of $E(M)$ where $A\in \cT_Z$
if $\lm(A)\leq k-2$ and $|Z\cap A|\leq k-2$. Then the following hold.
\begin{itemize}
\item[(i)] $\cT_Z$ is a tangle of order $k$ in $M$.
\item[(ii)] $M_{\cT_Z}|Z\cong U_{k-1,|Z|}$.
\end{itemize}
\end{lemma}

\begin{proof}
It follows from the definition of $\cT_Z$ that (T1) holds. 
Say $(A,B)$ is a partition of $E(M)$ with $\lm(A)\leq k-2$. By the definition of $k$-connected set,
either $|A\cap Z|\leq k-2$ or $|B\cap Z|\leq k-2$. Hence either $A$ or $B$ is in $\cT_Z$ so that
(T2) holds. 
If $A,B,C\in\cT_Z$, then $|A\cap Z|,|B\cap Z|,|C\cap Z|\leq k-2$. Hence 
$|(A\cup B\cup C)\cap Z|\leq 3k-6<3k-5\leq |Z|$. Hence $A\cup B\cup C\neq E(M)$ so that
(T3) holds. 

Say $e\in E(M)$. Then $\lm(\{e\})\leq 1\leq k-2$, and $|\{e\}\cap Z|\leq 1\leq k-2$. Hence
$\{e\}\in\cT$ so that, by (T3), $E(M)-\{e\}\notin \cT_Z$ and (T4) holds. This proves (i). Part (ii) is routine.
\end{proof}

If $Z$ satisfies the hypotheses of Lemma~\ref{get-k-tangle}, then we say that the tangle
$\cT_Z$ is the $k$-tangle in $M$ {\em generated} by $Z$. All up we have

\begin{lemma}
\label{all-up}
Let $\cT$ be a tangle in the matroid $M$ of order $k$ and breadth $t$. Then $M$ contains a 
$t$-element $k$-connected set $Z$ such that $M_\cT|Z\cong U_{k-1,|Z|}$. 
Moreover, if $k\geq 3$ and $t\geq 3k-5$, then $\cT$ is generated
by $Z$.
\end{lemma}

The connection between uniform submatroids of the tangle matroid and $k$-connected
sets outlined above clearly justifies the use of breadth as a parameter to measure the ``size''
of a tangle. 

\subsection*{More basic facts on tangle matroids}

Let $\cT$ be a tangle in a matroid $M$. 
Recall that a subset $A$ of $E(M)$ is $\cT$-weak if $A$ is contained in a $\cT$-small set. Note that 
$\cT$-weak sets can have arbitrarily high connectivity, so  $\cT$-weak sets are not 
necessarily $\cT$-small. To see this let $M$ be a matroid with two components $X$ and $Y$, each of which is 
highly connected.
Let $\cT$ be the tangle of order 2 that identifies the component $X$. Then every subset of
$Y$ is $\cT$-weak, even the ones with high connectivity. 
The next lemma follows immediately from the definitions.
By Theorem~\ref{tangle-matroid} if $\cT$ is a tangle of order $k$, then $r(M_\cT)=k-1$.

\begin{lemma}
\label{basic-stuff}
Let $\cT$ be a tangle of order $k$ in a matroid $M$ and suppose $A\subseteq E(M)$.  Then  
\begin{itemize}
\item[(i)] $A$  is $\cT$-weak
if and only if $r_{M_\cT}(A)<k-1$; and  
\item[(ii)] if $A$ is a proper flat of $M_\cT$, then $A$ is $\cT$-small and $r_{M_\cT}(A) = \lm_M(A)$.
\end{itemize}
\end{lemma}

Note that the converse of Lemma~\ref{basic-stuff}(ii) does not hold. Hall~\cite[Theorem 4.1, Corollary 4.2]{H15} proved the following result. Recall that a matroid $N$ is a {\em quotient} of a matroid $M$ if there exists a matroid
$M'$ and a subset $Z$ of $E(M')$ such that $M'\ba Z=M$ and $M'/Z=N$. 

\begin{lemma} 
\label{dennis2}
Let $\cT$ be a tangle in a matroid $M$. If $X \subseteq E(M)$, then 
$\cl_M(X) \subseteq \cl_{M_{\cT}}(X)$. Moreover, $M_{\cT}$ is a quotient of $M$.
\end{lemma}

\begin{corollary}
\label{full-closed}
Let $\cT$ be a tangle in a matroid $M$ and let $A$ be a  flat of $M_\cT$. Then 
$A$ is fully closed in $M$. 
\end{corollary}

\begin{proof}
Say $A$ is a flat of $M_\cT$. Then $A$ is closed in $M$ by Lemma~\ref{dennis2}. The fact
that $A$ is also coclosed follows from the same argument and the fact that $\cT$ is also a tangle in 
$M^*$.
\end{proof}

A matroid $M$ is {\em round} if its ground set cannot be covered by
two hyperplanes. Equivalently, $M$ is round if, whenever $(A,B)$ is a partition of $E(M)$,
 either $A$ or $B$ is spanning. The fact that the tangle matroid is round is an immediate consequence of 
 Theorem~\ref{dennis}. 

\begin{corollary}
\label{round}
Let $\cT$ be a tangle in a matroid $M$. Then  $M_\cT$ is round.
\end{corollary}

\begin{lemma}
\label{3-con}
Let $\cT$ be a tangle of order at least $3$ in  a $3$-connected matroid $M$. Then $M_\cT$ is 
$3$-connected.
\end{lemma}

\begin{proof}
Since $\cT$ has order at least 3, we have $r(M_\cT)\geq 2$. By Corollary~\ref{round}, 
$M_\cT$ is round. Hence $\si(M_\cT)$ is $3$-connected. Thus $\si(M_\cT)$ has at least three elements. 
Assume that $M_\cT$ has a set $X$ that is a loop or a nontrivial parallel class. Then $M$ has at least four elements. Thus, as $M$ is $3$-connected, it is simple. Then, since $r_{M_{\cT}}(X) \in \{0,1\}$, we have 
$\lm_M(B) \in \{0,1\}$ for some $B$ in $\cT$ such that $X\subseteq B$. As $B \in \cT$, we see that $|E(M) - B| \ge 2$, 
so $M$ is not $3$-connected, a contradiction.
\end{proof}

A set $X$ in a matroid $M$ is {\em solid} if there is no partition $\{A,B\}$ of $X$ with 
$\lm(A),\lm(B)<\lm(X)$.  Note that an exactly
$3$-separating set $X$ in a 3-connected matroid is solid if and only if $|X|\geq 3$.
The next lemma works for any 3-connected matroid, but it is the application for tangle matroids
that we need.

\begin{lemma}
\label{lines}
Let $\cT$ be a tangle of order at least $4$ in a $3$-connected matroid $M$. Let
$F$ be a solid proper flat of $M_\cT$ of rank at least two, and let $L$ be a solid rank-$2$ flat of
$M_\cT$ that is not contained in $F$. Then $|F\cap L|\leq 1$. Moreover, if $a\in F\cap L$,
then 
$a\in\cl_M(F-\{a\})$ and $a\in\cl_M(L-\{a\})$.
\end{lemma}

\begin{proof}
By Lemma~\ref{tangle-matroid-properties}(iii), $M_\cT$ is simple. Hence, as $L \not \subseteq F$, we see that $|F \cap L| \le 1$. 
By Lemma~\ref{dennis2}, both $F$ and $L$ are flats of $M$. 
For 
$a\in F\cap L$, since $L$ is solid flat of $M_\cT$, it follows that $|L| \ge 3$, so $a\in\cl_M(L-\{a\})$. Thus $a\in\cl_M(E(M) - F)$, so 
$a\in\cl_M(F-\{a\})$ otherwise $F$ is not solid. 
\end{proof}

\section{Tangles Generated in Minors} 

In what follows, we freely use the next elementary result (see, for example, \cite[Corollary 8.2.5]{O11}).

\begin{lemma}
\label{low-con}
Let $N$ be a minor of a matroid $M$ and $A\subseteq E(M)$. Then $\lm_N(A\cap E(N))\leq \lm_M(A)$.
\end{lemma}

Let $N$ be a minor of a matroid $M$ and let $\cT_N$ be a tangle of order $k$ in $N$. Now
let $\cT_M$ be the collection of all sets $A$ of $E(M)$ such that $\lm_M(A)<k-1$, and
$A\cap E(N)\in \cT_N$. The next lemma follows immediately from the definitions \cite[Lemma 5.1]{GGW09}.

\begin{lemma}
\label{induce-up}
$\cT_M$ is a tangle of order $k$ in $M$.
\end{lemma}

We say that $\cT_M$ is the {\em tangle in $M$ induced by $\cT_N$}. In the other direction, things
are not as smooth. If $\cT_M$ is a tangle of order $k$ in $M$, then it may be that there is
more than one tangle of order $k$ in $N$ that induces $\cT_M$, or there may be no tangle
of order $k$ in $N$ that induces $\cT_M$. 

Let $\cS$ be a collection of $(k-1)$-separating subsets of a matroid $M$. 
We say $\cS$ {\em generates a tangle $\cT$ in $M$} if $\cT$ is the unique tangle of order $k$  for which $\cS \subseteq \cT$.
Let $N$ be a minor of a matroid $M$, let $\cT_M$ be a tangle in $M$ of order $k$, and let 
$\cT_N$ be a tangle in $N$. 
We say that $\cT_M$ {\em generates the tangle $\cT_N$ in $N$}
if $\cT_N$ is the unique tangle of order $k$ in $N$ that contains $\{A\cap E(N):A\in \cT_M\}$.
The next lemma follows from the definitions.

\begin{lemma}
\label{induce-down}
Let $\cT_M$ be a tangle of order $k$ in a matroid $M$, and let $N$ be a minor of $M$. If $\cT_M$ generates the
tangle $\cT_N$ in $N$, then $\cT_N$ induces $\cT_M$ in $M$.
\end{lemma}

The next lemma enables us to focus on minors that arise from  deleting or contracting a single element.

\begin{lemma}
\label{transitive}
Let $N$ be a minor of a matroid $M$, and let $P$ be a minor of $N$. Let $\cT_M$, $\cT_N$, and
$\cT_P$ be tangles in $M$, $N$, and $P$, respectively. If $\cT_M$ generates $\cT_N$ in $N$, and
$\cT_N$ generates $\cT_P$ in $P$, then $\cT_M$ generates $\cT_P$ in $P$.
\end{lemma}

\begin{proof}
Let $\cS_{M,N}=\{A\cap E(N):A\in\cT_M\}$, let 
$\cS_{M,P}=\{A\cap E(P):A\in\cT_M\}$, and let $\cS_{N,P}=\{A\cap E(P):A\in\cT_N\}$. 

\begin{sublemma}
\label{transitive1}
$\cS_{M,P}\subseteq \cS_{N,P}$.
\end{sublemma}

\begin{proof}
Say $Z\in \cS_{M,P}$. Then $Z=A\cap E(P)$ for some $A\in \cT_M$.
Now $A\cap E(N)\in \cS_{M,N}$. Hence $A\cap E(N)\in\cT_N$. Thus $A\cap E(P)\in \cS_{N,P}$,
that is, $Z\in \cS_{N,P}$.
\end{proof}

Since $\cS_{M,P}\subseteq \cS_{N,P}$ we know that every member of $\cS_{M,P}$ is 
$\cT_P$-small. If $\cT_P$ is the only tangle of order $k$ in $P$ with this property, then
$\cS_{M,P}$ generates $\cT_P$, that is, $\cT_M$ generates $\cT_P$ in $P$, as required.

Assume otherwise. Then there is a tangle $\cT'$ of order $k$ in $P$ such that $\cT'\neq \cT_P$ and
such that every member of $\cS_{M,P}$ is $\cT'$-small. If every member of $\cS_{N,P}$ is
$\cT'$-small, then we contradict the fact that $\cS_{N,P}$ generates $\cT_P$. Hence there is a member
$A'$ of $\cS_{N,P}$ such that $A'$ is not $\cT'$-small. Now $A' = A_1 \cap E(P)$ for some $A_1$ in $\cT_N$. By 
Lemma~\ref{induce-down}, $\cT_N$ induces $\cT_M$, so $A_1 = A_0 \cap E(N)$ for some $A_0$ in $\cT_M$. Thus $A' = A_0 \cap E(P)$, 
so $A' \in \cS_{M,P}$. Hence $A'$ is  $\cT'$-small, a contradiction. 
\end{proof}

Recall that, for a tangle $\cT$, a set $W$ is $\cT$-weak if $W \subseteq A$ for some set $A$ in $\cT$.

\begin{lemma}
\label{add-weak}
Let $\cT_M$ be a tangle of order $k$ in a matroid $M$ and let $N$ be a minor of $M$. Then  
$\cT_M$ generates the tangle $\cT_N$ in $N$ if and only if the collection $\cW$ defined by
$\cW=\{W\subseteq E(N): \lm_N(W)\leq k-2;  \text{~$W$ is $\cT_M$-weak in $M$}\}$ generates $\cT_N$.
\end{lemma}

\begin{proof} 
Observe that $\cW$ contains $\{A\cap E(N): A\in\cT_M\}$.
Now let 
$\cT$ be a tangle in $N$ that contains $\{A\cap E(N):A\in \cT_M\}$. 
Say $W\in \cW$. Then $W$ is $\cT_M$-weak, so there exists $A\in \cT_M$ such that $W\subseteq A$.
Now $A\cap E(N)\in \cT$ and $W\subseteq A\cap E(N)$. By (T3a),  $W\in \cT$. 

We deduce that every tangle in $N$ that contains $\{A\cap E(N):A\in\cT_M\}$
also contains $\cW$. Thus if $\cW$ generates $\cT_N$, then $\cT_M$ generates $\cT_N$ and conversely.
\end{proof} 

Let $\cT$ be a tangle of order $k$
in a matroid $M$ and let $N$ be a minor of $M$. 
We are hoping that $\cT$ will generate a tangle of order $k$ in $N$. The orientation of many
$(k-1)$-separating sets in $N$ will be determined by Lemma~\ref{add-weak} but there will
typically be plenty of others. In the case that $N$ is a single-element removal of $M$, we can
be precise about what these undetermined sets are.

\begin{lemma}
\label{ambiguous}
Let $\cT$ be a tangle of order $k$ in a matroid $M$ and let $a\in E(M)$. Let $(X,Y)$
be a partition of $E(M/a)$ such that $\lm_{M/a}(X)\leq k-2$. Then the following hold.
\begin{itemize}
\item[(i)] At most one of $X$ and $Y$ is $\cT$-weak.
\item[(ii)] If neither $X$ nor $Y$ is $\cT$-weak, then $\lm_M(X)=\lm_M(Y)=k-1$,
and $a\in\cl_M(X)\cap \cl_M(Y)$. 
\end{itemize}
\end{lemma}

\begin{proof}
Assume that both $X$ and $Y$ are $\cT$-weak. Then there are $\cT$-small sets
$X'$ and $Y'$ containing $X$ and $Y$, respectively.  Now, provided $k \ge 3$, we see that $\{X',Y',\{a\}\}$ is a cover of $E(M)$
by $\cT$-small sets, a contradiction. If $k = 2$, then $\{X',Y'\}$ is a cover of $E(M)$
by $\cT$-small sets unless $X' = X$ and $Y' = Y$. In the exceptional case, each of $X$ and $Y$ is a union of components of $M$. 
Hence $\{a\}$ is a component of $M$ and, again, $\{X',Y',\{a\}\}$ is a cover of $E(M)$
by $\cT$-small sets.

Assume that neither $X$ nor $Y$ is $\cT$-weak. Say $\lm_M(X)\leq k-2$. Then
either $X$ or $Y\cup\{a\}$ is $\cT$-small, so  one of $X$ or $Y$ is $\cT$-weak.
Hence $\lm_M(X),\lm_M(Y)\geq k-1$. Since $\lm_{M/a}(X)=k-2$, we have
$\lm_M(X)=\lm_M(Y)=k-1$. If $a\notin\cl_M(X)\cap \cl_M(Y)$, then $\lm_{M/a}(X)=k-1$.
Hence $a\in\cl_M(X)\cap\cl_M(Y)$, as required.
\end{proof}

Partitions satisfying (ii) of Lemma~\ref{ambiguous} are the ones we have to focus on if we
are seeking a tangle generated by $\cT$ in $M/a$. 
The next lemma gives sufficient conditions that enable us to canonically orient all of the
$(k-1)$-separations of $M/a$. Such an orientation may fail to deliver a tangle, but, if it succeeds,
that tangle will be generated by $\cT$ in $M/a$. The lemmas that follow take advantage of a 
flat of the tangle matroid with certain properties. This may seem somewhat arbitrary.
We note that the flat arises in a natural way in applications later in the paper.

\begin{lemma}
\label{canon}
Let $\cT$ be a tangle of order $k$ in a matroid $M$, let $F$ be a flat of $M_\cT$ of rank $t\leq k-2$,
and let $a$ be an element of $F$ such that $F-\{a\}$ is solid in $M/a$ and 
$\lm_{M/a}(F-\{a\})=t$. Let $(X,Y)$
be a partition of $E(M/a)$ such that $\lm_{M/a}(X)=k-2$, and such that neither $X$ nor 
$Y$ is $\cT$-weak. Then, up to switching the labels of $X$ and $Y$,  the following hold where
$G=E(M)-F$.
\begin{itemize}
\item[(i)] $\lm_M(X),\lm_M(Y)=k-1$ and $a\in\cl_M(X) \cap \cl_M(Y)$; 
\item[(ii)] $\lm_{M/a}(F\cap X)\geq t$ and $\lm_{M/a}(F\cap Y)<t$;  
\item[(iii)] $\lm_{M/a}(G\cap X)> k-2$ 
and $\lm_{M/a}(G\cap Y)\leq k-2$; and 
\item[(iv)] $G\cap Y$ is
$\cT$-small.
\end{itemize}
Moreover, $k \ge 3$ and if $\cT'$ is a tangle in $M/a$ that induces $\cT$, then $Y$ is $\cT'$-small.
\end{lemma}

\begin{proof}
Part (i) follows from Lemma~\ref{ambiguous}. As $F-\{a\}$ is solid in $M/a$ and $\lm_{M/a}(F-\{a\})=t$,
we may assume up to labels that $\lm_{M/a}(F\cap X)\geq t$. By the symmetry of the connectivity function,
$\lm_{M/a}((F- \{a\})\cup X)=\lm_{M/a}(G\cap Y)$. Thus, by submodularity, we have
\begin{align*} 
\lm_{M/a}(G\cap Y) &\leq \lm_{M/a}(F- \{a\})+\lm_{M/a}(X)-\lm_{M/a}(F\cap X)\\
&\leq t+(k-2)-t.
\end{align*}
Hence $\lm_{M/a}(G\cap Y)\leq k-2$. 

Since $F$ is a rank-$t$ flat of $M_{\cT}$, it follows that $F$ is $\cT$-small and $\lm_M(F) = t$. Thus 
$\lm_M(F) = \lm_{M/a}(F- \{a\})$, so $a\notin\cl_M(G)$. Hence $a\notin\cl_M(G\cap Y)$ and 
$\lm_M(G\cap Y)=\lm_{M/a}(G\cap Y)\leq k-2$. If $G\cap Y$ is $\cT$-large, then $F\cup X$ is 
$\cT$-small. This implies that $X$ is $\cT$-weak, a contradiction. Hence $G\cap Y$ is
$\cT$-small.

By symmetry, the argument above proves that if $\lm_{M/a}(G\cap X)\leq k-2$, then $G\cap X$
is $\cT$-small. This implies that $\{F,G\cap X,G\cap Y\}$ is a cover of $E(M)$
by $\cT$-small sets contradicting the definition of a tangle. Hence 
$\lm_{M/a}(G\cap X)>k-2$.

If $\lm_{M/a}(F\cap Y)\geq t$ then, via the argument at the start of the proof, we deduce that
$\lm_{M/a}(G\cap X)\leq k-2$. Hence $\lm_{M/a}(F\cap Y)<t$. Thus $k \ge 3$. 

Now assume that $\cT'$ is a tangle in $M/a$ that induces $\cT$. The set $F-\{a\}$ is $\cT$-weak,
and $\lm_{M/a}(F-\{a\})=t\leq k-2$. Hence $F-\{a\}$ is $\cT'$-small. Moreover,  $G\cap Y$ is $\cT$-small
and is therefore $\cT'$-small. If $X$ is $\cT'$-small, then we can cover $E(M/a)$ by three $\cT'$-small
sets. Hence $X$ is $\cT'$-large, so $Y$ is $\cT'$-small.
\end{proof}

Let $\cT$ be a tangle of order $k$ in a matroid $M$, let $F$ be a rank-$t$ flat of $M_\cT$ where
$t\leq k-2$, and let $G=E(M)-F$. 
Assume that      $a$  is an element of $F$ for which $\lm_{M/a}(F-\{a\})=\lm_M(F)=t$,
and $F-\{a\}$ is solid in $M/a$. Let $(U,V)$ be a partition of $E(M/a)$ such that $\lm_{M/a}(U) \le k-2$. 
The $(k-1)$-separating partition $(U,V)$ is of {\it Type I} if $U$ or $V$ is $\cT$-weak; it is 
of {\it Type II} if neither $U$ nor $V$ is $\cT$-weak. By Lemma~\ref{ambiguous}, when $(U,V)$ is of Type I, 
exactly one of $U$ and $V$ is $\cT$-weak.

We now construct a set $\cT'$ of sets in $M/a$ that are {\it determined by $\cT$ in $M/a$} as follows. 
If $(U,V)$ is a  Type I $(k-1)$-separating partition  of $E(M/a)$, then the member of $\{U,V\}$ that is $\cT$-weak is in $\cT'$. 
If $(U,V)$ is a  Type II $(k-1)$-separating partition  of $E(M/a)$, then, by Lemma~\ref{canon}, there is a unique $Y$ in $\{U,V\}$ such that 
$\lm_{M/a}(Y\cap F)<t$ and $\lm_{M/a}(Y\cap G)\leq k-2$. This member $Y$ is in $\cT'$ and we have that 
$\lm_M(Y) = k-1$ and $a \in \cl_M(Y)$. We say that a member $Z$ of $\cT'$ is of Type I or Type II if it comes from a Type I or Type II 
$(k-1)$-separating partition  of $E(M/a)$.

\begin{corollary}
\label{selection}
Let $\cT$ be a tangle of order $k$ in a matroid $M$, let $F$ be a proper flat of $M_\cT$.
Assume that $a\in F$ is such that $\lm_{M/a}(F-\{a\})=\lm_M(F)$ and 
$F-\{a\}$ is solid in $M/a$. Let $\cT'$ be the collection of sets
determined by $\cT$ in $M/a$. Then the following hold.
\begin{itemize}
\item[(i)] If $(A,B)$ is a partition of $E(M/a)$ with $\lm_{M/a}(A)\leq k-2$, then exactly one
of $A$ and $B$ belongs to $\cT'$.
\item[(ii)] Say $A\in \cT'$. Then $\lm_{M/a}(A)\leq k-2$. Moreover, if $B\subseteq A$ and
$\lm_{M/a}(B)\leq k-2$, then $B\in\cT'$.
\item[(iii)] If $e\in E(M/a)$, then $\{e\}\in\cT'$.
\item[(iv)] If $\cT$ generates a tangle $\cT_{M/a}$ in $M/a$, then $\cT_{M/a}=\cT'$.
\end{itemize}
\end{corollary}

\begin{proof} As in Lemma~\ref{canon}, let
$t=\lm_{M/a}(F-\{a\})$ and let $G=E(M)-F$.
Part (i) follows immediately from the definition of the members of $\cT'$. 
Consider (ii). Say $A\in\cT'$. Then, by the definition of $\cT'$, we have $\lm_{M/a}(A)\leq k-2$.
Say  $B\subseteq A$ has $\lm_{M/a}(B)\leq k-2$. Suppose $A$ is of Type~I. 
Then $A$ is $\cT$-weak. As $B \subseteq A$, it follows that $B$ is $\cT$-weak. 
As $\lm_{M/a}(B)\leq k-2$, we deduce that $B$ is a Type I member of $\cT'$. 
We may now assume that $A$ is not of Type I. Then, with $A' = E(M/a) - A$, neither $A$ nor $A'$ is 
$\cT$-weak. Thus $A$ of Type II, and $A$ and $A'$ are $\cT$-strong.  Assume that $B \not\in \cT'$. 
Then $B$ is not $\cT$-weak. Let $B' = E(M/a) - B$.  Assume $B'$ is not $\cT$-weak. Then, by Lemma~\ref{canon}, $\lm_M(B) = \lm_M(B')= k-1$ and  
either $B$ is a Type II set, a contradiction; or $\lm_{M/a}(B' \cap F) < t$ and $\lm_{M/a}(B' \cap G) \le k-2$. Since we also know that $\lm_{M/a}(A\cap F) < t$, it follows by Lemma~\ref{set-diff} that 
$\lm_{M/a}((A - B') \cap F) < t$ or $\lm_{M/a}((B' - A) \cap F) < t$. Now, in $M/a$,  one of the partitions  
$\{(A - B') \cap F, B' \cap F\}$  and $\{(B' - A) \cap F, A\cap F\}$ of $F - \{a\}$ violates the fact that this set is solid. We conclude that $B'$ is $\cT$-weak. As $B' \supseteq A'$, we deduce that $A'$ is $\cT$-weak, a contradiction.

Consider (iii).  By Lemma~\ref{canon}, $k\geq 3$, so $r_{M_{\cT}}(\{e\})\leq k-1$ and then, 
by Lemma~\ref{basic-stuff}(i),
$e$ is $\cT$-weak. Moreover, $\lm_{M/a}(\{e\}) \le k-2$ so  $\{e\}$ is a Type I member of  $\cT'$.

For (iv), assume that $\cT$ generates a tangle $\cT_{M/a}$ in $M/a$.  
Let $(U,V)$ be a partition of $E(M/a)$ for which  $\lm_{M/a}(U) \le k -2$. If $U$ or $V$ is $\cT$-weak, then $(U,V)$ is a Type I $(k-1)$-separating partition  of $E(M/a)$ and the $\cT$-weak member of $\{U,V\}$ is in $\cT'$. Say $U \in \cT'$. Then 
$U \subseteq U_0$ where $U_0 \in \cT$. Thus $\lm_M(U_0) \le k-2$, so $\lm_{M/a}(U_0- \{a\}) \le k-2$. Hence $U_0- \{a\} \in \cT_{M/a}$. As $\lm_{M/a}(U) \le k -2$ and $U_0 - \{a\} \in \cT_{M/a}$, so $U \in \cT_{M/a}$. Thus if $U$ is a Type I member of $\cT'$, then 
$U \in \cT_{M/a}$

Now suppose that $U$ is a Type II member of $\cT'$. By the last part of Lemma~\ref{canon}, $U \in \cT_{M/a}$.  We deduce that $\cT' \subseteq \cT_{M/a}$. By (i)--(iii), $\cT'$ is a tangle provided there are no three members of $\cT'$ whose union is $E(M/a)$. But this holds because $\cT_{M/a}$ is a tangle. Since $\cT_{M/a}$ is the unique tangle of order $k$ in $M/a$ that contains $\{A - \{a\}: A \in \cT\}$, and $\cT'$ is a tangle of   order $k$ in $M/a$ that contains $\{A - \{a\}: A \in \cT\}$, we conclude that $\cT' = \cT_{M/a}$.
\end{proof}

The upshot of Corollary~\ref{selection} is that to prove that $\cT$ generates a tangle in $M/a$, 
it suffices to prove that $E(M/a)$ is not covered by three members of $\cT'$. One way to guarantee this
is to strengthen the condition on $F-\{a\}$ in $M/a$.
A subset $A$ of $E(M)$ is {\em titanic} if there is no partition $\{X,Y,Z\}$ of $A$
such that $\lm(X),\lm(Y),\lm(Z)<\lm(A)$. We can replace the condition in the definition of titanic
by an apparently weaker
one. An immediate consequence of the next result is that if $A$ is titanic, then it is solid.

\begin{lemma}
\label{tit-cover}
Let $M$ be a matroid and $A\subseteq E(M)$. Then $A$ is titanic if and only there are no
sets $X,Y,Z\subseteq A$ such that $X\cup Y\cup Z=A$ and $\lm(X),\lm(Y),\lm(Z)<\lm(A)$.
\end{lemma}

\begin{proof}
Assume that $X\cup Y\cup Z=A$ and that $\lm(X),\lm(Y),\lm(Z)<\lm(A)$. Say $X\cap Y \neq \emptyset$.
By Lemma~\ref{set-diff}, we may assume, up to labels that $\lm(X-Y)< \lm(A)$. 
Replacing $X$ by $X-Y$ and continuing this process produces a partition that proves 
that $A$ is titanic. The converse is immediate.
\end{proof}

\begin{lemma}
\label{tit-tangle}
Let $\cT$ be a tangle of order $k$ in a  matroid $M$, let $F$ be a rank-$t$ flat of $M_\cT$ where 
$t\leq k-2$.
Assume that $a\in F$, that $\lm_{M/a}(F-\{a\})=\lm_M(F)$, and that
 $F-\{a\}$ is titanic in $M/a$. Then 
$\cT$ generates a tangle of order $k$ in $M/a$.
\end{lemma}

\begin{proof}
Let $G=E(M)-F$. 
Let $\cT'$ be the collection of sets that are determined by $\cT$ in $M/a$. Assume that $\cT'$ does not generate a tangle in 
$M/a$. Then there are sets $X_1,X_2,X_3$ in $\cT'$ such that $X_1\cup X_2\cup X_3=E(M/a)$.
Every member of $\cT'$ is either of Type~I, that is, it is contained in a $\cT$-small set, or it is of Type~II.

\begin{sublemma}
\label{tit-tangle-0}
If $X_i$ is of Type~I, then there is a $\cT$-small
subset $X_i'$ of $E(M)$  that contains $X_i$ such that $\lm_M(X_i'\cap F)<t$.
\end{sublemma}

\begin{proof} 
Since $X_i$ is of Type~I, there is a $\cT$-small
subset $X_i'$ of $E(M)$  that contains $X_i$. 
Say $\lm_M(X'_i \cap F)\geq t$. Now $\lm_M(F)=t$ and $\lm_M(X'_i)\leq k-2$. Thus, as 
$$\lm_M(X'_i\cap F)+\lm_M(X'_i\cup F)\leq \lm_M(X'_i)+\lm_M(F),$$
it follows that $\lm_M(X_i'\cup F)\leq k-2$.
Assume that $X_i'\cup F$ is $\cT$-large. Then $\{E(M)-(X'_i\cup F), F, X'_i\}$ is a cover of $E(M)$
by $\cT$-small sets, a contradiction. Therefore  $X'_i\cup F$ is $\cT$-small.

For simplicity, assume that $X'_1\cup F$ is $\cT$-small.
For each $i$ in $\{2,3\}$, if $X_i$ is of Type~I, let $X'_i$ be a $\cT$-small set containing $X_i$; and, if 
 $X_i$ is of Type~II,  let $X'_i=X_i\cap G$. Using Lemma~\ref{canon}(iv), we deduce that, in each case, $X'_i$ is $\cT$-small.
Then $\{X'_1\cup F,X'_2,X'_3\}$ is a cover of $E(M)$ by $\cT$-small sets, a contradiction.
Hence $\lm_M(X_i'\cap F)<t$ as required.
\end{proof}

For $i\in\{1,2,3\}$, if $X_i$ has Type I, then we take $X_i''$ to be the set $X_i'$ given by (\ref{tit-tangle-0}),
and if $X_i$ has Type II, then we take $X_i''$ to be $X_i$. As the sets $X_1''\cap (F-\{a\})$, $X_2''\cap (F-\{a\})$,
$X_3''\cap (F-\{a\})$ cover $F-\{a\}$, these sets provide a contradiction to the assumption that 
$F-\{a\}$ is titanic in $M/a$.





\end{proof}

\section{Preserving Breadth}

Until further notice, $\cT$ is a tangle of order $k\geq 2$ in a matroid $M$, the set $F$ is a flat of
$M_\cT$ with $r_M(F)=t\leq k-2$, and $a\in F$ has the properties that $\lm_{M/a}(F-\{a\})=t$
and that $F-\{a\}$ is titanic in $M/a$. Observe that $\lm(F)=t$. By Lemma~\ref{tit-tangle}, 
$\cT$ generates a tangle
$\cT_a$ in $M/a$. 

Let $M_1$ and $M_2$ be matroids on a common ground set $E$. We say that $M_1$
is {\em freer} than $M_2$ if $r(M_1)=r(M_2)$ and every set that is independent in $M_2$ is 
independent in $M_1$.  Equivalently, the identity map on $E$ is a rank-preserving weak map from $M_1$ to $M_2$. The next lemma is elementary 
and can be derived, for example, by combining Proposition 7.3.11 and Corollary 7.3.13 of \cite{O11}.

\begin{lemma}
\label{free-and-easy}
The matroid $M_1$ is freer than $M_2$ if $r(M_1)=r(M_2)$ and every hyperplane of $M_1$
is contained in a hyperplane of $M_2$.
\end{lemma}

We are interested in the relationship between $M_\cT\ba a$ and $M_{\cT_a}$. 
Note that, because a tangle in $M$ is a tangle in $M^*$,  the matroids $M_{\cT}$ and $(M^*)_{\cT}$ are equal. 
Thus $M_{\cT}\ba a = (M^*)_\cT\ba a$. This may cause confusion because of  the familiar identity that $M^*\ba a = (M/a)^*$.

\begin{lemma}
\label{freer}
The matroid $M_\cT\ba a$ is freer than $M_{\cT_a}$.
\end{lemma}

\begin{proof}
By Corollary~\ref{round},  $M_\cT$ is round, so $a$ is not a coloop of $M_\cT$. Hence $r(M_\cT\ba a)=k-1$.
By definition, $r(M_{\cT_a})=k-1$.  Let $H$ be a hyperplane of
$M_\cT\ba a$.  Then either $H$ or $H\cup a$ is $\cT$-small as hyperplanes of the tangle matroid
are $\cT$-small. Thus, by the definition of generating, $H$ is $\cT_a$-small.    
We deduce that $H$ is contained in a hyperplane of $M_{\cT_a}$ and
the lemma follows from Lemma~\ref{free-and-easy}.
\end{proof}

It would be very surprising if the breadth went up.

\begin{lemma}
\label{breadth-down}
For $M_{\cT}$ and $M_{\cT_a}$,
$$\breadth(M_{\cT_a})\leq \breadth(M_\cT).$$
\end{lemma}

\begin{proof}
By Lemma~\ref{freer},  $M_\cT\ba a$ is freer than $M_{\cT_a}$, so any uniform restriction of 
$M_{\cT_a}$ is also uniform in $M_{\cT}\ba a$.
\end{proof}

The real  difference between $M_\cT\ba a$ and $M_{\cT_a}$ is that elements
 of $F-\{a\}$ are potentially occupying more specialised positions
in $M_{\cT_a}$. It is useful to know that the ranks of certain sets are unchanged.
In particular, we have the following where $G=E(M)-F$.

\begin{lemma}
\label{static}
Suppose $X\subseteq E(M)-\{a\}$.
\begin{itemize}
\item[(i)] If $F-\{a\}\subseteq X$, then $X$ has the same rank in both $M_\cT\ba a$ and 
$M_{\cT_a}$.
\item[(ii)] If $X\subseteq G$, then $X$ has the same rank in both $M_\cT\ba a$ and 
$M_{\cT_a}$.
\end{itemize}
\end{lemma}

\begin{proof}
By Lemma~\ref{tit-tangle}, $\cT_a$ has order $k$ in $M/a$, so $r(M_{\cT_a})=k-1$. 
Since $M_{\cT}$ is round, it has no coloops. Hence $r(M_{\cT}\ba a)=r(M_\cT)=k-1$.
If $a$ is a loop of $M$, then it is readily verified that $M_{\cT_a}=M_{\cT}\ba a$ and the 
lemma holds. We assume from now on that $a$ is not a loop of $M$.

Assume that $X$ has different ranks in $M_{\cT}\ba a$ and $M_{\cT_a}$. Since $M_{\cT}\ba a$
is freer that $M_{\cT_a}$ and the two matroids have the same rank, we see that

\begin{sublemma}
\label{stat-0}
$X$ is not spanning in $M_{\cT_a}$.
\end{sublemma}

Since $X$ is not spanning in $M_{\cT_a}$, by Lemma~\ref{tangle-matroid-properties}, we have
$r_{M_{\cT_a}}(X)=\min\{\lm_{M/a}(Z):Z\supseteq X\text{~and  $Z$ is $\cT_a$-small}\}$.
Assume that $X'\supseteq X$ is a $\cT_a$-small set with $\lm_{M/a}(X')=r_{M_{\cT_a}}(X)$.
Say $\lm_{M/a}(X')=s-1$. By (\ref{stat-0}), $s\leq k-1$. Let $Y'=E(M)-(X'\cup\{a\})$. The next
claim follows from an application of Lemma~\ref{tikokino} and the assumption that
$r_{M_{\cT_a}}(X)\neq r_{M_\cT\ba a}(X)$.

\begin{sublemma}
\label{stat-1}
$\lm_M(X')=\lm_M(X'\cup\{a\})=s$, $r_{M_{\cT}\ba a}(X)=s$,  $a\in\cl_M(X')$ and $a\in\cl_M(Y')$. 
\end{sublemma}

\begin{sublemma}
\label{stat-2}
$a\notin\cl_M(G)$.
\end{sublemma}

\begin{proof} Since $\lambda_M(F)=\lambda_{M/a}(F-a)$ and $G=E(M)-F$, 
it follows by Lemma~\ref{tikokino} that $a\not\in {\rm cl}(G)$.
\end{proof}

\begin{sublemma}
\label{stat-3}
$F-\{a\}\nsubseteq X'$ and hence (i) holds.
\end{sublemma}

\begin{proof}
Say $F-\{a\}\subseteq X'$. Then $Y'\subseteq G$. By (\ref{stat-2}), $a\notin\cl_M(Y')$ contradicting 
(\ref{stat-1}). Hence $F-\{a\}\nsubseteq X'$. If $F-\{a\}\subseteq X$, then $F-\{a\}\subseteq X'$. Hence (i) holds.
\end{proof}

It remains to prove (ii). From now on we assume that $X\subseteq G$.

\begin{sublemma}
\label{stat-4}
$\lm_{M/a}(X'\cap (F-\{a\}))<t$.
\end{sublemma}

\begin{proof}
Assume that $\lm_{M/a}(X'\cap (F-\{a\}))\geq t$. Then, by submodularity 
$\lm_{M/a}(X'\cup (F-\{a\}))\leq s-1$ so that $\lm_M(X'\cup F)\leq s$. Consider the partition 
$(X'\cup (F-\{a\}),\{a\},Y'\cap G)$ of $E(M)$. We have $a\in\cl(X'\cup (F-\{a\}))$. By (\ref{stat-2}), 
$a\notin\cl(Y'\cap G)$. Hence $a\in\cl^*(X'\cup (F-\{a\}))$. By Lemma~\ref{up-down}
$\lm_M(X'\cup F)<\lm_M(X'\cup (F-\{a\}))$, so that $\lm_M(X'\cup F)\leq s-1$.

The set $X'\cup F$ is either $\cT$-small or $\cT$-large. Assume the latter.
Then $G\cap Y'$ is $\cT$-small so that $\{G\cap Y',X',F\}$ is a cover of $E(M)$ by
$\cT$-small sets, contradicting the fact that $\cT$ is a tangle. Hence
$X'\cup F$ is $\cT$-small. This implies that
$r_{M_\cT}(X)\leq \lm_M(X'\cup F)\leq s-1$, contradicting the assumption that
$r_{M_\cT}(X)=s$. The claim follows from this contradiction.
\end{proof}

\begin{sublemma}
\label{stat-5}
$\lm_M(X'\cap G)<s$.
\end{sublemma}

\begin{proof}
Assume that $\lm_M(X'\cap G)\geq s$.  By submodularity we have
$\lm_M(X'\cup G)\leq t$. Since $a\in\cl_M(X')$, we have $a\in\cl_M(X'\cup G)$.
Hence $\lm_M(X'\cup G\cup\{a\})\leq t$. 
By symmetry $\lm_M(Y'\cap F)=\lm_M(X'\cup G\cup\{a\})$.
Hence $\lm_M(Y'\cap F)\leq t$. By (\ref{stat-4}) and the fact that
$F-\{a\}$ is titanic in $M/a$, we have $\lm_{M/a}(Y'\cap F)\geq t$. Therefore $\lm_M(Y'\cap F)=t$.

Since $\lm_M(X'\cup G)\leq t$, we have $\lm_M((Y'\cap F)\cup\{a\})\leq t$. 
Hence $\lm_M((Y'\cap F)\cup\{a\})\leq \lm_M(Y'\cap F)$ so that $a\in\cl_M(Y'\cap F)$
or $a\in\cl^*_M(Y'\cap F)$. But $a\in\cl_M(X')$, so, by orthogonality, $a\notin\cl^*_M(Y'\cap F)$.
Hence $a\in\cl_M(Y'\cap F)$. Therefore $\lm_{M/a}(Y'\cap F)<t$ contradicting the fact that 
$F-\{a\}$ is titanic in $M/a$. Thus $\lm_M(X'\cap G)<s$ as required. 
\end{proof}

We have $X\subseteq X'\cap G$ and $\lm_M(X'\cap G)<s$. Since $X'$ is $\cT$-small, so too is
$X'\cap G$. This implies that $r_{M_\cT}(X)<s$ contradicting the fact that
$r_{M_\cT}(X)=s$. The lemma follows from this final contradiction.
\end{proof}

\section{Low-Rank Flats of the Tangle Matroid}

The earlier results apply in general with no assumption having been made about the rank
of the flat $F$ of the tangle matroid. For this paper, we need to consider the cases when
$F$ has rank at most 2 and that is the focus of this section.

Note that Lemmas~\ref{ambiguous}, \ref{canon}, \ref{selection}, \ref{tit-tangle}, \ref{freer},
\ref{breadth-down} and \ref{static} all concern tangles in a minor obtained by 
contracting an element of a matroid. The reason for the focus on contraction was that
it facilitated more natural geometric arguments in proofs. Since the tangles in matroids are invariant
under duality, each of the above mentioned lemmas has an obvious dual which concerns tangles
in a minor obtained by deleting an element of the matroid. In what follows we apply the dual
version of the lemmas.

\subsection*{Rank-$0$ flats of the tangle matroid.}

\begin{lemma}
\label{loops-away}
Let $\cT$ be a tangle of order $k\geq 2$ in a matroid $M$. If $a$ is a loop of 
$M_\cT$, then $\cT$ generates a tangle $\cT_a$ of order $k$ in $M\ba a$. Moreover,
$\breadth(\cT_a)=\breadth(\cT)$.
\end{lemma}

\begin{proof}
Let $F$ be the unique rank-0 flat of $M_\cT$, that is, $F$ is the set of loops of 
$M_\cT$. Observe that $F-\{a\}$ is titanic in $M\ba a$, even when $F-\{a\}$ is empty
as there are no sets in a matroid whose rank is less than zero. Suppose first that $k = 2$. Then 
$M$ has a connected component $X$ with at least two elements such that 
$\cT = \{A \subseteq E(M): \lm(A)=0,A \cap X = \emptyset\}$. Moreover, $F = E(M) - X$. Then 
$\{A \subseteq E(M\ba a): \lm_{M\ba a}(A)=0,A \cap X = \emptyset\}$ is a tangle $\cT_a$ of order $2$ in 
$M\ba a$ that is generated by $\cT$. 
On the other hand, when $k \ge 2$, Lemma~\ref{tit-tangle} gives that 
$\cT$ generates a tangle $\cT_a$ in $M\ba a$. For arbitrary $k\ge 2$, let $U$ be the ground set of a maximal
spanning uniform submatroid of $M_\cT$. Then $U\subseteq E(M)-F$. By 
Lemma~\ref{static}(ii), $M_\cT|U=M_{\cT_a}|U$. Hence $\breadth(\cT_a)=\breadth(\cT)$.
\end{proof}

\begin{lemma}
\label{in-a-component}
Let $\cT$ be a tangle of order $k\geq 2$ in a matroid $M$  and let $F$ be the 
set of loops of $M_{\cT}$. Then $M_\cT\ba F$ is connected. Moreover,
$\cT$ generates a tangle $\cT'$ in $M\ba F$ for which $\breadth(\cT')=\breadth(\cT)$.
\end{lemma}

\begin{proof}
If $M_\cT\ba F$ is not connected, then $M_\cT$ is not round, a contradiction to Corollary~\ref{round}. Thus $M_\cT\ba F$ is indeed
connected. The remainder of the lemma follows by repeated application of Lemma~\ref{loops-away}
and Lemma~\ref{transitive}.
\end{proof}

A tangle $\cT$ of order $k$ in a matroid $M$ is {\em breadth-critical} if,  whenever $N$ is a proper minor of $M$ and $\cT$ generates a tangle $\cT'$ of order $k$ in $N$, we have 
$\breadth(\cT')<\breadth(\cT)$.
The next corollary is immediate.

\begin{corollary}
\label{get-con}
If $\cT$ is a breadth-critical tangle in a matroid $M$, then $M$ is connected.
\end{corollary}

\subsection*{Rank-$1$ flats of the tangle matroid.}
If $\cT$ is a tangle of order at least 3 in a connected matroid $M$, then $M_\cT$ is loopless. It follows
that rank-1 flats of $M_\cT$ are parallel classes. Put in other words, the 
parallel classes of $M_\cT$ are the maximal $\cT$-small 2-separating sets of $M$.
The next lemma is clear.

\begin{lemma}
\label{2-tit-1}
Let $F$ be a $2$-separating set of a connected matroid $M$. Then $F$ is titanic.
\end{lemma}

\begin{lemma}
\label{2-tit-2}
Let $\cT$ be a tangle of order at least $3$ in a connected matroid $M$ and let 
$F$ be a $\cT$-small $2$-separating set of $M$ with $|F|\geq 2$. Assume that
$a\in F$ and $M\ba a$ is connected.
Then $\cT$ generates a tangle $\cT_a$ in $M\ba a$. Moreover, 
$M_\cT\ba a=M_{\cT_a}$ and
$\breadth(\cT_a)=\breadth(\cT)$.
\end{lemma}

\begin{proof}
We may assume that $F$ is a maximal $\cT$-small 2-separating set in $M$.
Then $F$ is a rank-1 flat of $M_\cT$. Since $M$ is connected, $M_\cT$ is loopless.
Hence $F$ is a parallel class of $M_\cT$. By Lemma~\ref{2-tit-1}, 
$F-\{a\}$ is titanic in $M\ba a$. Hence, by Lemma~\ref{tit-tangle}, $\cT$ generates a tangle 
$\cT_a$ in $M\ba a$. Since $M\ba a$ is connected, $M_{\cT_a}$ is loopless.
By Lemma~\ref{freer}, $M_\cT\ba a$ is freer than $M_{\cT_a}$, so 
$F-\{a\}$ is a  set of parallel elements in  $M_{\cT_a}$.

Say $X\subseteq E(M)-\{a\}$. By Lemma~\ref{static}, $X$ has the same rank in 
both $M_\cT\ba a$ and $M_{\cT_a}$ unless both $X\cap F$ and $F-(X \cup \{a\})$ are nonempty.
In the exceptional case, since $F- \{a\}$ is a set of parallel elements in each of $M_{\cT \ba a}$ and $M_{\cT_a}$, 
we see that $r_{M_{\cT_a}}(X)=r_{M_{\cT_a}}(X\cup (F - \{a\}))=r_{M_\cT\ba a}(X\cup (F - \{a\}))
=r_{M_\cT\ba a}(X)$. We deduce that $M_\cT\ba a=M_{\cT_a}$. 

Since $a$ is a member of a non-trivial parallel class of $M_{\cT}$, there is a maximal spanning 
uniform restriction $U$ of $M_\cT$ that avoids $a$. Since $M_\cT\ba a=M_{\cT_a}$, we see that 
$U$ is a maximal spanning uniform restriction of $M_{\cT_a}$. 
Hence $\breadth(\cT_a)=\breadth(\cT)$.
\end{proof}

\begin{corollary}
\label{2-tit-3}
Let $\cT$ be a tangle of order $k \ge 3$ in a matroid $M$.  If $M$ is not 
$3$-connected, then $M$ has an element $a$ such that, for some $N$ in $\{M\ba a, M/a\}$, 
the tangle $\cT$ generates 
a $k$-tangle $\cT'$ in $N$ with  $\breadth(\cT')=\breadth(\cT)$.
\end{corollary}

\begin{proof}
Assume that $M$ is not $3$-connected.
Then, for some $t$ in $\{0,1\}$, there is a    partition $(X,Y)$  of $E(M)$ with $\lm(X)=t$ and $|X|,|Y|>t$. 
By the definition of a tangle, we may assume that $X\in\cT$. If $t = 0$, then $r_{M_{\cT}}(X) = 0$ and taking $a$ in $X$, the result follows by 
Lemma~\ref{loops-away}. Thus we may assume that $t = 1$, so $M$ is connected. 
Then, for $a$ in $X$, 
 by a well-known  result of Tutte~\cite{wT66}, either $M\ba a$ or $M/a$ is connected. 
We lose no generality in assuming that $M\ba a$ is connected.
By Lemma~\ref{2-tit-2}, $\cT$ generates a tangle $\cT_a$ in $M\ba a$ such
that $\breadth(\cT_a)=\breadth(\cT)$. 
\end{proof}

\subsection*{Rank-$2$ flats of the tangle matroid.} 

Let $\cT$ be a tangle of order at least 4 in a matroid $M$. Assume that 
$\cT$ is breadth-critical. Then, by Corollaries~\ref{get-con} and \ref{2-tit-3}, $M$ is 
$3$-connected. By Corollary~\ref{round} and Lemma~\ref{3-con}, $M_\cT$ is 3-connected and round. 
Our goal is to bound the size of a rank-2 flat of $M_\cT$.
Say that $F$ is such a flat. Then $F$ is a maximal $\cT$-small $3$-separating set.
By Corollary~\ref{full-closed}, $F$ is fully closed in $M$. We first note an obvious lemma.

\begin{lemma}
\label{3-tit}
Let $F$ be an exactly $3$-separating set in a $3$-connected matroid $M$. Then
$F$ is titanic if and only if $|F|\geq 4$.
\end{lemma}

\begin{proof}
Say $A\subseteq F$ and $\lm(A)<\lm(F)$. Then, since $M$ is 3-connected,
$|A|\leq 1$. The lemma follows from this observation.
\end{proof}

The next lemma relies on some more definitions. Let $M$ be a matroid. Elements $x$ and $y$ of $M$ are {\em clones}
if the function that interchanges $x$ and $y$ and fixes every  element of $E(M) - \{x,y\}$
is an automorphism of $M$. An element $z$ of $M$ is {\em fixed in $M$} if there is no
single-element extension $M'$ of $M$ by an element $z'$ with the property that $z$ and $z'$  are
clones in $M'$ and $\{z,z'\}$ is independent in $M'$.
Say $Z\subseteq E(M)$. Then an element
$z\in Z$ is {\em freely placed} on $Z$ if $z\in\cl(Z-\{z\})$ and, whenever $C$
is a circuit of $M$ containing $z$, the closure of $C$ contains $Z$. 

Our interest is in a special case of rank-2 flats in $3$-connected matroid,
and we focus on that. We omit the straightforward proof of the next result.

\begin{lemma}
\label{fixed}
Let $F$ be a rank-$2$ flat of a $3$-connected matroid $M$ where $|F|\geq 3$.
\begin{itemize}
\item[(i)] If $a\in F$, then $a$ is freely placed on $F$ if and only if $a$ is not fixed in $M$.
\item[(ii)] If $a\in F$, then $a$ is fixed in $M$ if and only if $M$ has a flat $A$ containing $a$ such that
$a\in\cl(A-\{a\})$ and $F\cap A=\{a\}$.
\item[(iii)] If $a$ and $b$ are distinct elements of $F$, then $a$ and $b$ are clones 
in $M$ if and only
if both $a$ and $b$ are freely placed on $F$.
\end{itemize}
\end{lemma}

Recall that, in a matroid $M$, the interior, $\inter_M(X)$, of a set $X$ is $X - (\cl_M(E(M) - X) \cup \cl^*_M(E(M) - X))$.

\begin{lemma}
\label{int-free}
Let $\cT$ be a tangle of order at least $4$ in a  
$3$-connected matroid $M$, and let $F$ be a maximal 
$\cT$-small $3$-separating set of $M$ with at least three elements. If $a\in \inter_M(F)$,
then $a$ is freely placed on the rank-$2$ flat $F$ in $M_\cT$.
\end{lemma}

\begin{proof}
Certainly $F$ is a rank-2 flat of $M_\cT$. Assume that $a$ is not freely placed on $F$ in 
$M_\cT$. By Lemma~\ref{fixed}(i) and (ii), $M_\cT$ has a flat $A$ of $M_\cT$ such that
$a\in\cl_{M_\cT}(A-\{a\})$ and $A\cap F=\{a\}$. Say $r_{M_\cT}(A)=t$.
Clearly $A$ is proper, so, by Lemma~\ref{basic-stuff}(ii), $\lm_M(A)=r_{M_\cT}(A)=t$.
We also have that $r_{M_\cT}(A-\{a\})=t$. Hence $\lm_M(A-\{a\})\geq 
\lm_M(A)$. It follows that either $a\in\cl_M(A-\{a\})$, in which case, 
$a\in\cl_M(E(M)-F)$; or $a\in\cl^*_M(A-\{a\})$, in which case, $a\in\cl^*(E(M)-F)$. Both 
cases imply that $a\notin\inter_M(F)$.
\end{proof}

Let $\cT$ be a tangle in the matroid $M$. We say that a subset $U$ of $E(M)$ 
is a {\em witness for $\breadth(\cT)$} if $M_\cT|U$ is a maximal spanning uniform restriction
of $M_\cT$.

\begin{lemma}
\label{free-breadth}
Let $\cT$ be a tangle of order at least four in a $3$-connected matroid $M$ and let 
$F$ be a rank-$2$ flat of $M_\cT$ with at least three elements. Let $U$ be a witness
for $\breadth(\cT)$. Then  
\begin{itemize}
\item[(i)] $|U\cap F|\leq 2$.
\item[(ii)] For $a\in U\cap F$ and $b\in F-U$, if $b$ is freely placed on $F$ in $M_\cT$,
then $(U-\{a\})\cup \{b\}$ is also a witness for $\breadth(\cT)$.
\end{itemize}
\end{lemma}

\begin{proof}
Since $\cT$ has order at least four, the rank of $M_\cT$ is at least three, so $r(U)\geq 3$. 
A uniform matroid of rank at least three cannot contain a triangle. Hence $|U\cap F|\leq 2$.

Say $a\in U\cap F$ and $b\in F-U$ where $b$ is freely placed on $F$ in $M_\cT$. 
Let $U'=(U-\{a\})\cup \{b\}$. Assume that $U'$ is not a 
witness for $\breadth(\cT)$. Then $M_\cT|U'$ is not a uniform matroid, so it contains a non-spanning
circuit $C$ that must contain $b$. But $b$ is freely placed on $F$, so 
$F\subseteq \cl_{M_\cT}(C-\{ b\})$. Hence $a\in\cl_{M_\cT}(C-\{b\})$, so  
$(C-\{b\})\cup\{a\}$ contains a circuit $C'$. But $C'\subseteq U$ and $C'$ does not span $U$.
This contradicts the assumption that $M_\cT|U$ is a uniform matroid.
\end{proof}

We are now able to prove lemmas that provide sufficient conditions for an element to 
be removed from our rank-2 flat $F$ while preserving the breadth of a tangle.

\begin{lemma}
\label{keep-int}
Let $\cT$ be a tangle of order $k\geq 4$ in a $3$-connected matroid $M$, let
$F$ be a maximal $\cT$-small $3$-separating set of $M$, and say $a\in F$. Assume that 
the following hold.
\begin{itemize}
\item[(i)] $M\ba a$ is $3$-connected.
\item[(ii)] $\lm_M(F)=\lm_{M\ba a}(F-\{a\})$.
\item[(iii)] $|F|\geq 5$.
\item[(iv)] $\inter_{M\ba a}(F-\{a\})\neq\emptyset$.
\end{itemize}
Then $\cT$ generates a tangle $\cT_a$ of order $k$ in $M\ba a$. Moreover, 
$\breadth(\cT_a)=\breadth(\cT)$.
\end{lemma}

\begin{proof}
By assumption $F$ is 3-separating in $M$. It is exactly 3-separating because $M$ is 3-connected, $|F|\geq 5$ and 
$F$ is $\cT$-small, so that $|E(M)-F|\geq 2$.  By Lemma~\ref{3-tit},  
$F-\{a\}$ is titanic in $M\ba a$.
It now follows from Lemma~\ref{tit-tangle} that $\cT$ generates a tangle $\cT_a$ in
$M\ba a$.

We have $\inter_{M\ba a}(F-\{a\})$ is nonempty, so, by Lemma~\ref{2int}, 
there are  distinct elements
$b$ and $c$ in $\inter_{M\ba a}(F-\{a\})$. Hence 
$b,c\in\inter_M(F)$. 

\begin{sublemma}
\label{keep-int-1}
There is a witness $U$ for $\breadth(\cT)$ with the property that
$U\cap F\subseteq \{b,c\}$. 
\end{sublemma}

\begin{proof} Let $U$ be a witness for $\breadth(\cT)$.
By Lemma~\ref{int-free}, $b$ and $c$ are freely placed on $F$ in $M_{\cT}$. 
By Lemma~\ref{free-breadth}(i), $|F\cap U|\leq 2$. 
 Then, by Lemma~\ref{free-breadth}(ii), we may assume that $a \not\in U$ and $b \in U$. 
The assertion holds if $|U \cap F| = 1$, or if $U \cap F = \{b,c\}$ so we may assume that $U \cap F = \{b,d\}$ for $d \neq c$. Then, by Lemma~\ref{free-breadth}(ii) again, 
$(U - \{d\}) \cup \{c\}$ is a witness for $\breadth(\cT)$ that contains $\{b,c\}$. Thus the assertion holds.
\end{proof}

By Lemma~\ref{breadth-down}, $\breadth(\cT)\geq \breadth(\cT_a)$. Suppose  the lemma fails. Then 
$M_{\cT_a}|U\neq M_\cT|U$. Since $M_\cT\ba a$ is freer than $M_{\cT_a}$, 
we deduce that   that $M_{\cT}|U$ is freer than $M_{\cT_a}|U$. Hence there is a
circuit $C$ of $M_{\cT_a}$ such that $C\subseteq U$ and $C$ is independent in 
$M_{\cT}\ba a$. By Lemma~\ref{static}(ii), $C\cap F\neq \emptyset$. 

Let $C'$ denote the closure of $C$ in $M_{\cT_a}$. Assume
that $F-\{a\}\subseteq C'$.  Then, by Lemma~\ref{static}(i)
$r_{M_\cT\ba a}(C')=r_{M_{\cT_a}}(C')$. But this implies that $C$ is dependent in
$M_\cT$. Hence $F-\{a\}$ is not contained in $C'$.

We now know that $C$ contains exactly one element of $F-\{a\}$. Since $U \cap F \subseteq \{b,c\}$, 
we may assume that $b\in C$. Hence $b$ is not freely placed on $F-\{a\}$ in $M_{\cT_a}$.
But $b\in\inter_{M\ba a}(F-\{a\})$. This contradicts Lemma~\ref{int-free}.
\end{proof}

No doubt the next lemma is well known.

\begin{lemma}
\label{keep-3-con}
Let $M$ be a $3$-connected matroid and $F$ be a fully closed set with $\lm(F)=2$
and $|F|\geq 4$. If $x\in\guts(F)$, then $M\ba x$ is $3$-connected.
\end{lemma}

\begin{proof}
Assume that the lemma fails. Let $G=E(M)-F$. Since $F$ is fully closed, $|G|\geq 3$. 
Since $x \in \guts(F)$, we see that $(G,F-\{x\})$ is a $2$-separation of $M/x$.
But $|G|,|F-\{x\}|\geq 3$, so,   
by Bixby's Lemma, $M\ba x$ is $3$-connected up to series pairs.  
Thus $x$ is in a triad of $M$. Let $T$ be such a triad.

Since $F$ is fully closed and $x\in\guts(F)$, we have that $x \in F \cap \cl(G)$. 
Thus, by orthogonality, $T \not \subseteq F$. Moreover, as $F$ is fully closed, $|T \cap G| \neq 1$. 
We deduce that $|T \cap G| = 2$, so $x \in \cl^*(G)$. As $x \in \cl(G)$, Lemma~\ref{up-down}(i) implies that 
$\lm(G \cup \{x\}) < \lm(G)$, contradicting the fact that $M$ is $3$-connected.
\end{proof}

\begin{lemma}
\label{guts-away}
Let $\cT$ be a tangle of order $k\geq 4$ in a $3$-connected matroid $M$ and let
$F$ be a maximal $\cT$-small $3$-separating set of $M$. Assume that 
$|F|\geq 5$ and that $|\guts(F)|\geq 3$. Then the following hold.
\begin{itemize}
\item[(i)] If $x\in \guts(F)$, then $\cT$ generates a tangle $\cT_x$ of order $k$ in $M\ba x$.
\item[(ii)] If $x\in \guts(F)$, then $M_{\cT_x}=M_{\cT}\ba x$.
\item[(iii)] In $\guts(F)$, there is an element $a$ such that $\breadth(\cT_a)=\breadth(\cT)$ where $\cT_a$
is the tangle of order $k$ in $M\ba a$ generated by $\cT$.
\end{itemize}
\end{lemma}

\begin{proof}
Say $x\in\guts(F)$. By Lemma~\ref{keep-3-con}, $M\ba x$ is 3-connected. 
Since $|F|\geq 5$ and $M\ba x$ is $3$-connected, $F-\{x\}$ is titanic in $M\ba x$. Thus, by
Lemma~\ref{tit-tangle}, $\cT$ generates a tangle $\cT_x$ of order $k$ in $M\ba x$. Thus (i) holds.

Now $M\ba x$ is $3$-connected, so, by Lemma~\ref{3-con} $M_{\cT_x}$ is 3-connected. 
From this   and the fact that $F-\{x\}$ has rank 2 in both $M_{\cT}\ba x$ and $M_{\cT_x}$, 
we deduce that 

\begin{sublemma}
\label{guts-away-0} 
$M_{\cT}|(F-\{x\})=M_{\cT_x}|(F-\{x\})$. 
\end{sublemma}

Assume that $M_{\cT}\ba x\neq M_{\cT_x}$. By Lemma~\ref{freer},
$M_{\cT}\ba x$ is freer than $M_{\cT_x}$, so there is a circuit $C$ of $M_{\cT_x}$ that is
independent in $M_\cT\ba x$. By (\ref{guts-away-0}) and
Lemma~\ref{static}(ii), $C$ meets both $E(M)-F$ and $F-\{x\}$. 
Let $C'=\cl_{M_{\cT_x}}(C)$. If $F-\{x\}\subseteq C'$, then, by Lemma~\ref{static}(i),
$C'$ has the same rank in both $M_\cT\ba x$ and $M_{\cT_x}$. Hence $F-\{x\}\not \subseteq C'$. 
As $r_{M_\cT}(F - \{x\}) = 2$, we see that $|C' \cap (F - \{x\})| < 2$. But 
$|C \cap (F - \{x\})| \ge 1$, so there is a unique element $c$ in $C' \cap (F - \{x\})$. Thus 
$C - \{c\} \subseteq C' - \{c\} \subseteq G$. Hence, by Lemma~\ref{static}(ii), 

\begin{sublemma}
\label{guts-away-1} 
$r_{M_{\cT}\ba x}(C' - \{c\}) = r_{M_{\cT_x}}(C' - \{c\}) = r_{M_{\cT_x}}(C- \{c\}) = r_{M_{\cT_x}}(C).$
\end{sublemma}

\begin{sublemma}
\label{guts-away-2}
$C'-\{c\}$ is a flat of $M_{\cT}\ba x$.
\end{sublemma}

\begin{proof}
Assume that this fails. Then there is an element $d$ not in $C' - \{c\}$ such that 
$r_{M_\cT\ba x}((C' - \{c\}) \cup \{d\}) =  r_{M_\cT\ba x}(C' - \{c\})$. As $M_{\cT}\ba x$ is freer than $M_{\cT_x}$, we see by (\ref{guts-away-1}) that 
$r_{M_{\cT_x}}((C' - \{c\}) \cup \{d\}) =  r_{M_{\cT_x}}(C' - \{c\})$, so $d = c$. Thus, by  (\ref{guts-away-1}) again, $r_{M_\cT\ba x}(C') = r_{M_{\cT_x}}(C).$
Thus $r_{M_{\cT_x}}(C) = r_{M_{\cT}\ba x}(C') \ge r_{M_{\cT}\ba x}(C) \ge r_{M_{\cT_x}}(C).$ Hence equality holds throughout and we have a contradiction as $C$ 
is a circuit in $M_{\cT_x}$ but an independent set in $M_\cT\ba x$.
\end{proof}

Let $r_{M_{\cT_x}}(C') = t$.  By Lemma~\ref{basic-stuff}, $C'$ is a maximal $\cT_x$-small
$(t+1)$-separating set in $M\ba x$. Since $c\in\cl_{M_{\cT_x}}(C'-\{c\})$, it follows that 
$\lm_{M\ba x}(C'-\{c\})\geq \lm_{M\ba x}(C')$. Hence
 $c\in\cl_{M\ba x}(C'-\{c\})$ or $c\in\cl^*_{M\ba x}(C'-\{c\})$.
In the former case, $c\in\cl_{M_\cT}(C'-\{c\})$ by Lemma~\ref{dennis2}, 
contradicting the fact that $C'-\{c\}$ is a flat of $M_\cT\ba x$.

We now know that $c\in\cl^*_{M\ba x}(C'-\{c\})$. Then, as $C' - \{c\} \subseteq E(M) - F$, we deduce that  $c\in\cl^*_{M\ba x}(E(M)-F)$.
This implies that $c\in\coguts_{M\ba x}(F-\{x\})$. We are now in a position to obtain a
contradiction to Lemma~\ref{guts-coguts}. Since $F$ is fully closed in $M$, it follows 
from Lemma~\ref{keep-fcl} that $F- \{x\}$ is fully closed in $M\ba x$.
As $|\guts(F)| \ge 3$, there are at least 
two elements in $\guts_{M\ba x}(F-\{x\})$. As $M\ba x$   is $3$-connected, we obtain a 
contradiction to Lemma~\ref{guts-coguts}.
Hence (ii) holds.

Let $U$ be a witness for $\breadth(\cT)$. As $\cT$ has order $k\geq 4$,
the matroid $M_{\cT}$ has rank at least three, so   $U$ has rank at least three. Such a 
uniform matroid cannot contain a triangle, so $|F\cap U|\leq 2$. Let
$a$ be an element of $\guts_M(F)-U$. By (ii), $(M_\cT\ba a)|U=M_{\cT_a}|U$.
Hence $U$ is also a witness for $\breadth(\cT_a)$. We conclude that
$\breadth(\cT)=\breadth(\cT_a)$, as required.
\end{proof}

\section{Proof of the Main Theorem}

Recall that a matroid $M$ is {\em weakly $4$-connected} if
$M$ is 3-connected and a subset  $A$ of $E(M)$ has $\lm(A)=2$ 
only if $|A|\leq 4$ or $|E(M)-A|\leq 4$. The purpose of this section is to establish the following theorem.

\begin{theorem}
\label{breadth-crit}
Let $\cT$ be a tangle of order $k\geq 4$ in a matroid $M$. Then 
\begin{itemize}
\item[(i)] $M$ is weakly $4$-connected; or 
\item[(ii)] $M$ has an element $a$ such that, for some $N$ in $\{M\ba a, M/a\}$, 
the tangle $\cT$ generates an order-$k$ tangle $\cT'$ in $N$ with $\breadth(\cT') = \breadth(\cT).$ 
\end{itemize}
\end{theorem}

The following is an immediate consequence of this theorem. 

\begin{corollary}
\label{breadth-crit-2}
Let $\cT$ be a tangle of order $k\geq 4$ in a matroid $M$. If $\cT$ is
breadth-critical, then $M$ is weakly $4$-connected.
\end{corollary}

We begin with some lemmas. The next result 
follows from Bixby's Lemma \cite{bixby} (see \cite[Lemma~8.7.3]{O11}).

\begin{lemma}
\label{bixby}
Let $e$ be an element of the $3$-connected matroid $M$. If $(A,B)$ is a $2$-separation of 
$M/e$ with $|A|,|B|\geq 3$, and $M\ba e$ is not $3$-connected, then $e$ is in a triad of $M$.
\end{lemma}

\begin{lemma}
\label{interior}
Let $F$ be a fully closed exactly $3$-separating set in a $3$-connected matroid $M$. 
Suppose that $r(F)>2$, that $r^*(F)>2$, and that $|F|\geq 5$. Then
there is an element $a\in F$ and a  $3$-connected matroid $N$ in 
$\{M\ba a,M/a\}$ such that $|\inter_N(F-\{a\})|\geq 2$.
\end{lemma}

\begin{proof}
We first consider the case that $r(F)=3$. Let $G=E(M)-F$.

\begin{sublemma}
\label{interior-0}
If $r(F)=3$, then $r_M(G)=r(M)-1$ and $r^*_M(F)\geq 4$.
\end{sublemma}

\begin{proof}
We have $\lm_M(F)=2$ and $F$ is exactly 3-separating, so that
$r(G)=r(M )+2-r_M(F)=r(M)-1$.

We have $r^*_M(F)=r_M(G)+|F|-r(M)$ (see for example \cite[Proposition~2.19]{O11}). Since
$|F|\geq 5$, we have $r^*_M(F)\geq r(M)-1+5-r(M)=4$ as required.
\end{proof}

\begin{sublemma}
\label{interior-1}
If $r(F)=3$, and $F$ contains an element $a$ such that $M\ba a$ is $3$-connected,
then $|\inter_{M\ba a}(F-\{a\})|\geq 2$.
\end{sublemma}

\begin{proof}
Say $r(F)=3$. Suppose that $a\in F$ and $M\ba a$ is 3-connected. Since $M\ba a$ is 3-connected and 
$|F-\{a\}|\geq 4$, we have $\lm_{M\ba a}(F-\{a\})\geq 2$. But $\lm_{M\ba a}(F-\{a\})\leq \lm_M(F)=2$, 
so equality holds. Hence
\begin{align*}
r_{M\ba a}(F-\{a\})&=r(M\ba a)+2-r_{M\ba a}(G)\\
&=r(M)+2-(r(M)-1)=3.
\end{align*}
By (\ref{interior-0}) $r^*_{M\ba a}(F-\{a\})\geq 3$. 
This means that   outcome (iii) or outcome (iv) of Lemma~\ref{interior-decoration} holds. In both
cases, $|\inter_{M\ba a}(F-\{a\})|\geq 2$ as desired.
\end{proof}

\begin{sublemma}
\label{interior-1.1}
If $r(F)=3$, and $a\in F\cap \cl_M(G)$, then $M\ba a$ is $3$-connected so that
$|\inter_{M\ba a}(F-\{a\})|\geq 2$.
\end{sublemma}

\begin{proof}
Assume $a\in F\cap \cl_M(G)$ and $M\ba a$ is not 3-connected. Note that 
$r_M(F-\{a\})=3$, as otherwise $M$ is not 3-connected. Thus $a\in\cl_M(F-\{a\})$.
We deduce that $\lm_{M/a}(F-\{a\})=1$. But $|F-\{a\}|,|G|\geq 3$ so that,
by Lemma~\ref{bixby} $a$ is in a triad $\{a,b,c\}$. Since $a\in\cl_M(F-\{a\})$
we have $(F-\{a\})\cap\{b,c\}\neq \emptyset$. Say $b\in F-\{a\}$. Then
$c\in\cl^*(\{a,b\})$ so that $c\in\cl^*(F)$. Since $F$ is fully closed $c\in F$.
But $\{a,b,c\}\subseteq F$. Hence $a\in\cl^*(F-\{a\})$.

We have $a\in\cl(F-\{a\})$ and $a\in\cl^*(F-\{a\})$, so $\lm_M(F)<\lm_M(F-\{a\})$.
Since $a\in\cl_M(G)$, we have $\lm_M(F-\{a\})\leq \lm_M(F)=2$. This implies the 
contradiction that
$\lm_M(F)<\lm_M(F)$.

Hence $M\ba a$ is $3$-connected and $|\inter_{M\ba a}(F-\{a\})|\geq 2$.
\end{proof}

\begin{sublemma}
\label{interior-1.2}
If $F$ is a cocircuit of $M$, then there is an element $a\in F$ such that
$M\ba a$ is $3$-connected so that $|\inter_{M\ba a}(F-\{a\})|\geq 2$.
\end{sublemma}

\begin{proof}
Assume that $M\ba x$ is not 3-connected for all $x\in F$.
As $r(F) = 3$ and $|F|\geq 5$, there is a circuit $C$ contained in $F$. 
By a theorem of Lemos~\cite{L89}, 
$M$ has a triad $T^*$ meeting $C$. By orthogonality and the fact that $F$ is fully closed, 
we deduce that $T^* \subseteq F$. This is  a contradiction as 
both $T^*$ and $F$ are cocircuits, and $|F|\geq 5$.
\end{proof}

If there is no element $a\in F\cap \cl_M(G)$, then $F$ is a cocircuit of $M$. It now follows
from (\ref{interior-1.1}) and (\ref{interior-1.2}) that $|\inter_{M\ba a}(F-\{a\})|\geq 2$.

\begin{sublemma}
\label{interior-2}
The lemma holds if $r(F)> 3$ and $r^*(F)> 3$. 
\end{sublemma}

\begin{proof}
If there is an element $x$ of $F$ such that  $N \in \{M\ba x, M/x\}$ and $N$ is $3$-connected, 
then $F- \{x\}$ is not a line of 
$N$ or of $N^*$, so outcome (iii) or outcome (iv) of Lemma~\ref{interior-decoration} 
holds for $N$  and the lemma holds. 
Thus we may assume that if $x \in F$, then neither $M\ba x$ nor $M/x$
is 3-connected. As $F$ is fully closed in $M$, it is fully closed in all $N$ in $\{M\ba x, M/x\}$   
for all $x$ in $F$. 
This gives a contradiction to a result of Oxley~\cite[Theorem 1.1]{O03}. 
\end{proof}

If $F$ satisfies the hypotheses of the lemma, then we are either in the case
of (\ref{interior-1}), the dual of (\ref{interior-1}), or (\ref{interior-2}). 
\end{proof}

\begin{proof}[Proof of Theorem~\ref{breadth-crit}]
Let $\cT$ be a tangle of order $k\geq 4$ in  $M$. By Corollary~\ref{2-tit-3}, 
$M$ is  $3$-connected. 
 Assume that
$M$ is not weakly $4$-connected. Then $M$ has   a 
$3$-separation $(X,Y)$   with $|X|,|Y|\geq 5$. We may assume that $X$ is $\cT$-small.
Then $r_{M_\cT}(X)=2$. Let $F=\cl_{M_\cT}(X)$. Then, since $r(M_{\cT}) = k-1 \ge 3$, we see that 
$F\neq E(M)$. By Corollary~\ref{full-closed}, $F$ is fully closed in $M$. 

Assume that $r_M(F)=2$. As $M$ is 3-connected, $E(M)-F$ spans $F$. 
Then $|\guts_M(F)|=|F|> 3$. By Lemma~\ref{guts-away}, 
for some $a$ in $F$, the tangle $\cT$ generates a tangle $\cT_a$ of order $k$
in $M\ba a$ with $\breadth(\cT_a)=\breadth(\cT)$. Thus we may assume that $r_M(F)\ge 3$. 
By  duality, we may also assume that $r^*_M(F)\geq 3$.

Lemma~\ref{interior} now gives us that  there is an element
$a$ in $F$ and a matroid $N$ in $\{M\ba a,M/a\}$ such that $N$ is 3-connected and $|\inter_N(F-\{a\})|\geq 2$.
By Lemma~\ref{keep-int}, $\cT$ generates an order-$k$ tangle
$\cT_a$ in $N$ with $\breadth(\cT_a)=\breadth(\cT)$. 
\end{proof}

We are now in a position to prove Theorem~\ref{biggy} which we restate as a corollary of earlier
results of this section.

\begin{corollary}
\label{find-minor}
Let $\cT$ be a tangle of order $k\geq 4$ in a matroid $M$. Then $M$ has a weakly
$4$-connected minor $N$ with a tangle $\cT'$ of order $k$ such that 
$\cT'$ is generated by $\cT$ and $\breadth(\cT')=\breadth(\cT)$.
\end{corollary}

\begin{proof}
If $M$ is weakly $4$-connected, let $N=M$. Otherwise, by repeated application of
Theorem~\ref{breadth-crit}, there is a, necessarily finite, sequence 
$N_1,N_2,\ldots,N_m$ of matroids and a sequence $\cT_1,\cT_2,\ldots,\cT_m$ of order-$k$ tangles  
such that all of the following hold.
\begin{itemize}
\item[(i)] $M = N_1$;  
\item[(ii)] each 
$N_i$ for $i > 1$ is a single-element deletion or a single-element contraction of $N_{i-1}$;   
\item[(iii)] $\cT_1=\cT$ and if $i>1$, then $\cT_i$ is a tangle of $N_i$ that is generated by $\cT_{i-1}$ with 
$\breadth(\cT_i)=\breadth(\cT_{i-1})$;  and  
\item[(iii)] $N_m$ is weakly $4$-connected.
\end{itemize}
By Lemma~\ref{transitive}, $\cT_m$ is generated by $\cT$ in $N_m$, so the result holds. 
\end{proof}

\section{Tangles of Order $4$}
Until now, we have presented our main results for tangles of order at {\em least} 4. If we
are focussed on a ``4-connected component'' of our matroid, then it is a tangle of order
{\em exactly\ } 4 that we are interested in.  By Corollary~\ref{find-minor}, a tangle of order 4 in a 
matroid generates a tangle of order 4 in a weakly 4-connected minor that preserves its
breadth. In what follows, we make some observations about tangles in this world.

\begin{lemma}
\label{one-tangle}
Let $M$ be a weakly $4$-connected matroid with at least thirteen elements. Then
$M$ has a unique tangle of order $4$.
\end{lemma}

\begin{proof}
Let $\cT$ consist of those subsets $A$ of $E(M)$ for which $\lm(A)\leq 2$ and $|A|\leq 4$. It is easily seen that
$\cT$ is a tangle in $M$. Say $\cT'$ is another order-$4$ tangle  in $M$. Then there is a
set $X$ with $|X|\leq 4$ such that $E(M)-X$ belongs to $\cT'$. Let $(Y,Z)$ be a partition
of $X$ into sets with $|Y|,|Z|\leq 2$. At least one of $Y$ or $Z$ must be $\cT'$-large,
otherwise we cover $E(M)$ by three $\cT'$-small sets. Assume that $Y$ is $\cT'$-large.
Since $Y$ is $\cT'$-large, $|Y|=2$ by (T4). Say $Y=\{y_1,y_2\}$. Then
$\{E(M)-Y,\{y_1\},\{y_2\}\}$ is a cover of $E(M)$ by $\cT'$-small sets, a contradiction. 
\end{proof}

Assume that $\cT$ is a tangle of order 4 in a matroid $M$ and let $N$  be a
weakly 4-connected minor of $M$ such that the unique tangle of order~$4$ in $N$ is generated
by $\cT$ and has breadth equal to $\breadth(\cT)$. Then we say that $N$ is a {\em witness for $\cT$}.

Let $\{\cT_1,\cT_2,\ldots,\cT_s\}$ be the collection of tangles of order~$4$ in a matroid $M$.
This is a collection of incomparable tangles and can therefore be displayed in a tree-like way \cite{GGW09}.
Then there is a collection 
$\{N_1,N_2,\ldots,N_s\}$ of minors of $M$ such that, for each
$i\in\{1,2,\ldots,s\}$, the minor $N_i$ is a witness for $\cT_i$. 
Put together, we have a weak analogue of the $2$-sum decomposition of 
a matroid with its collection of $3$-connected minors. 

The 3-connected minors associated with the 2-sum decomposition are unique up to isomorphism,
but it is evident that a tangle of order~4 in $M$ can have non-isomorphic witnesses.
Also, given the 2-sum decomposition of a matroid, we can build the original matroid from its underlying
3-connected minors. Finding an analogue of this for tangles seems ambitious.  Utilising the 
$3$-separation tree of a $3$-connected matroid described by \cite{OSW04} and the results of 
\cite{CX12}, it is possible that something could be done in the case of representable matroids.

We now consider the structure of tangle matroids.

\begin{lemma}
\label{identity}
Let $P$ be a simple rank-$3$ matroid that cannot be covered by three lines.
Then $P$ has a unique $4$-tangle $\cT$. Moreover,  $M_{\cT}=P$.
\end{lemma}

\begin{proof}
Let $\cT$ consist of those subsets $A$ of $E(M)$ for which $r(A) \le 2$. 
Then one easily checks that $\cT$ is a tangle of order $4$. 
Assume that $\cT'$ is a tangle of order 4 in $M$. Say $A\subseteq E(M)$ and $r(A)\leq 2$.
If $|A|\leq 1$, then $A$ is $\cT'$-small. Assume that $|A|>1$ and that all proper subsets of 
$A$ are $\cT'$-small. Say $a\in A$. If $A$ is not $\cT'$-small, then
$\{\{a\},A-\{a\}, E(M)-A\}$ is a cover of $E(M)$ by $\cT'$-small sets.
Hence $A$ is $\cT'$-small. Therefore $\cT\subseteq\cT'$ and it follows that
$\cT=\cT'$.

By Theorem~\ref{dennis}, $P$ is a tangle matroid. By Lemma~\ref{dennis2}, $M_{\cT}$ is a quotient of $P$. As $r(M_{\cT}) = 3 = r(P)$, we deduce that $M_{\cT} = P$.
\end{proof}

The next result follows by combining the last two lemmas.

\begin{corollary}
\label{all-tangle-matroids}
Let $P$ be a matroid with at least thirteen elements. Then $P$ is the tangle matroid of the tangle of 
order $4$ associated with a weakly $4$-connected matroid $M$ if and only 
$P$ is simple, $r(P) = 3$, and each line of $P$ has at most four elements.
\end{corollary}

\begin{proof} 
Let $M$ be a weakly $4$-connected matroid with at least thirteen elements. By Lemma~\ref{one-tangle}, $M$ has a unique tangle $\cT$ of order 4. 
The tangle matroid $M_{\cT}$ is simple and has rank three. By Theorem~\ref{dennis}, $M_{\cT}$ cannot be covered by three lines.
Suppose $M_{\cT}$ has a line $L$ with at least five points. Then $\lm_M(L) = 2$. But, since $E(M) - L$ cannot be covered by two lines of $M_{\cT}$, it follows that 
 $|E(M) - L| \ge 5$. Then $(L, E(M) - L)$ is a $3$-separation of $M$ that contradicts the fact that $M$ is weakly $4$-connected. We deduce that each line of $M_{\cT}$ 
 has at most four elements.
 
 Conversely, let $P$ be a simple rank-$3$ matroid in which each line has most four elements. If $|E(P)| \ge 13$, then $E(P)$ cannot be covered by three lines. Thus, by 
 Lemma~\ref{identity}, $P$ has a unique tangle $\cT$ of order $4$ and $M_{\cT} = P$. Now $P$ is certainly $3$-connected. Moreover, because each line has at most four elements, 
 $P$ is weakly $4$-connected.
\end{proof}

\begin{lemma}
\label{breadth-root}
Let $M$ be a weakly $4$-connected matroid with at least thirteen elements and let
$\cT$ be the tangle of order $4$ associated with $M$.
Then $\breadth(\cT)\geq \sqrt{|E(M)|}$.
\end{lemma}

\begin{proof}
Let $\breadth(\cT)=\beta$ and let  
 $U$ be a witness of $\breadth(\cT)$.
Then $U$ is a spanning uniform restriction of the 
3-connected rank-3 matroid $M_{\cT}$. All other elements of $E(M)$ must lie on lines spanned
by pairs of elements of $U$. There are $\beta\choose 2$ such pairs, and each associated line has at most
four elements. We deduce that $|E(M)| \le \beta +2{\beta\choose 2}=\beta^2$. 
\end{proof}

We do not know the best-possible bound that can be given on $\breadth(\cT)$ in Lemma~\ref{breadth-root}. 

\section{Can we do Better?} \label{better}
It is natural to ask whether the connectivity condition on $M$ in Theorem~\ref{breadth-crit} can be strengthened beyond $M$ being weakly 4-connected. In this section, we 
present an example that shows that Theorem~\ref{breadth-crit} is in some sense best possible.

Say $s\geq 6$, and let $E=\{e_1,e_2,\ldots,e_s\}$. Let $M_1$ be a matroid on $E$ 
with $M_1\cong U_{3,s}$. Let $M_2$ be isomorphic to $M(K_4)$ and have ground set 
$\{a,b,c,d,e,f\}$ where $\{c,d,e\}$, $\{a,b,e\}$, $\{b,c,f\}$, and $\{a,d,f\}$ are triangles.
Observe that $\{e,f\}$ is not contained in a triangle. Consider $M_1\oplus M_2$. 
Extend this matroid by placing elements $f_1$ and $f_2$ freely on the lines 
$\{e,e_1\}$ and $\{f,e_2\}$, respectively. Extend the resulting matroid by placing 
elements $g_1$ and $g_2$ freely on the flats $E\cup\{f_1,e\}$
and $E\cup\{f_2,f\}$, respectively. Finally, delete the 
elements $e$ and $f$ to obtain  a matroid $M$ with ground set 
$\{a,b,c,d,f_1,f_2,g_1,g_2,e_1,\ldots,e_s\}$.

It is readily checked that $M$ has no triangles. Indeed $M$ 
is weakly $4$-connected with $|E(M)|=s+8$. Let $\cT$ denote  the
unique tangle of order $4$ in $M$.  All maximal $\cT$-weak 3-separating sets are 
pairs except the $4$-element $3$-separating set $\{a,b,c,d\}$. It follows that,  
for any pair $\{x,y\}\subseteq \{a,b,c,d\}$, we have  $M_{\cT}\ba x,y\cong U_{3,s+6}$. In other 
words, $\cT$ has breadth $s+6$. 

Consider the matroids $M/a$ and $M\ba a$. We shall show that $\cT$ generates tangles in
both of these matroids. First focus on $M/a$. Let $\cT'$ denote the unique tangle of order $4$ in $M/a$. 
Then $\cT'$ contains $\{A - {a}: A \in \cT\}$. Hence $\cT$ generates $\cT'$. 
The triangles $\{b,f_1,e_1\}$ and $\{d,f_2,e_2\}$ of $M/a$ guarantee that
a uniform restriction of $M_{\cT'}$ contains at most two elements of each of these sets. 
Hence $\breadth(\cT')\leq |E(M)/a|-2=s+5<\breadth(\cT)$. Now focus on $M\ba a$. Let 
$\cT''$ denote the unique tangle of order $4$ in $M\ba a$. Since $\cT''$ contains $\{A - {a}: A \in \cT\}$, it follows that 
$\cT''$ is generated by $\cT$ in $M\ba a$. The sets $\{b,f_1,g_1\}$ and
$\{d,f_2,g_2\}$ are $\cT''$-weak triads of $M\ba a$. Hence they are triangles of
$\cT''$. Arguing just as in the previous case, we deduce that $\breadth(\cT'')< \breadth(\cT)$.

The bijection on $E(M)$ that interchanges $a$ and $d$, interchanges $b$ and $c$, and fixes every other element is an automorphism of $M$; 
so is the  bijection that   interchanges $c$ and $d$, interchanges $a$ and $b$, and fixes every other element. 
Furthermore it is readily verified that $M$ is breadth critical. 

\begin{proposition}
\label{weak-stuck}
There exists a breadth-critical tangle of order $4$ in a weakly $4$-connected matroid that has a 
$4$-element $3$-separator. 
\end{proposition}

The previous example was based on a 4-element set that is both a circuit and a cocircuit. 
It is also possible to construct examples for graphic
matroids where the
4-element 3-separator is a fan. Let $G$ be a graph constructed as
follows. Begin with 
a simple 4-connected graph $H$ with no triangles that has 
a stable set $\{5,6,7,8\}$ of vertices. Let
$\{1,2,3,4\}$ be an additional set of vertices and add the edges $\{25,26,36,37,47,48,24,12,13,14\}$.
Then $M(G)$ is weakly 4-connected with a fan $\{12,13,14,24\}$ and, apart from that fan, all 
fully-closed 3-separators have size at most 2. 
Let $\cT$ be the unique tangle of order 4 in $M(G)$. Then
it is readily verified that $\cT$ is a breadth-critical tangle in $M(G)$.

\section{Back to $k$-connected Sets} 
\label{back}

We now return to our original assertion about $k$-connected sets in matroids. The next theorem is a 
restatement of Theorem~\ref{get-weak}.

\begin{theorem}
\label{k-connected}
Let $k\geq 4$ be an integer and $M$ a matroid with an $n$-element $k$-connected set where 
$n\geq 3k-5$.
Then $M$ has a weakly $4$-connected minor with an $n$-element $k$-connected set.
\end{theorem}

\begin{proof}
By Lemma~\ref{get-k-tangle}, $M$ has a tangle of order $k$ and breadth at least $n$. 
By Theorem~\ref{breadth-crit},
$M$ has a weakly $4$-connected minor $N$ with a tangle $\cT$ of order $k$ and breadth 
$m\geq n$. By the definition of breadth, $N$ has an $m$-element set $Z$ such that $N_\cT|Z\cong U_{k-1,m}$.
By Lemma~\ref{get-k-con}, $Z$ is a $k$-connected set in $N$.
\end{proof}

\section{Discussion}
A tangle $\cT$ of order $k$ in a matroid $M$ identifies a ``$k$-connected component'' of
$M$ and we have used the notion of breadth to measure the size of such a component. Thus,
if we are interested in measuring the ``size'' of a $4$-connected component, we need 
a tangle of order exactly 4. Nonetheless, Theorem~\ref{breadth-crit} is potentially of interest for larger
values of $k$. Let $t\geq 0$ be an integer, and let $(s_0,s_1,\ldots,s_t)$ be a sequence of
non-negative integers. Then a matroid $M$ is {\em $(s_0,s_1,\ldots,s_t)$-connected} if,
whenever $F\subseteq E(M)$ has $\lm(F)=i$ for $i\in\{0,1,\ldots,t\}$, 
either $|F|\leq s_i$ or $|E(M)-F|\leq s_i$. 

In this terminology, a matroid is weakly 4-connected if and only if
it is $(0,1,4)$-connected.
Thus we have proved that if $\cT$ is a breadth-critical tangle of order at least
4 in a matroid $M$, then $M$ is $(0,1,4)$-connected. We conjecture the following.

\begin{conjecture}
\label{con-1}
There is an infinite sequence $(s_0,s_1,s_2,\ldots)$ such that,  for 
all $k\geq 2$, if 
$\cT$ is a breadth-critical tangle of order at least $k$ in a matroid $M$, then 
$M$ is $(s_0,s_1,\ldots,s_{k-2})$-connected.
\end{conjecture}

For a stronger conjecture, one might speculate as to what a suitable sequence could be.
Observe that a $4$-separating set $F$ in a $(0,1,4)$-connected matroid is guaranteed to
be titanic if $|F|\geq (3\times 4)+1=13$. Define the sequence
$(t_0,t_1,t_2,\ldots)$ by $t_0=0$,  and, otherwise, $t_i=3t_{i-1}+1$. 
Note that $t_i=(3^i-1)/2$. 

\begin{conjecture}
\label{con-2}
If $\cT$ is a breadth-critical tangle of order at least $k$ in a matroid $M$,
then $M$ is $(t_0,t_1,\ldots, t_{k-2})$-connected.
\end{conjecture}

Let $\cT$ be an order-$k$ tangle  in a matroid $M$; say $t\in\{2,3,\ldots,k-1\}$. Then it
is easily seen that the collection 
$T_t(\cT)=\{A\in\cT:\lm(A)\leq t-2\}$ is a tangle of order $t$ in $M$. We say that
$T_t(\cT)$ is the {\em truncation of $\cT$ to order $t$}. Truncations of tangles
correspond to truncations of their tangle matroids.

\begin{lemma}
\label{truncate}
Let $\cT$ be a tangle of order $k$ in a matroid $M$, say $t\in\{2,3,\ldots,k-1\}$
and let $T_t(\cT)$ denote the truncation of $\cT$ to order $t$. 
Then $M_{T_t(\cT)}$ is the truncation to rank $t-1$ of $M_{\cT}$.
\end{lemma}

Since truncations of uniform matroids are uniform, it follows from Lemma~\ref{truncate}
that, if $\cT$ has order $k$, then $\breadth(T_t(\cT))\geq \breadth(\cT)$ for any truncation
$T_t(\cT)$, but it is easily seen that the converse does not hold.

Via truncation, we have a suite of tangles associated with a given tangle. For each member of
this suite of order at least 4, we can find a $(0,1,4)$-connected matroid that preserves its breadth.
Can we do this simultaneously?

\begin{conjecture}
\label{sim-breadth}
Let $\cT$ be a tangle of order $k\geq 4$ in a matroid $M$. For each 
$i\in\{4,5,\ldots,k\}$, let $T_i(\cT)$ denote the truncation of $\cT$ to order $i$.
Then there is a $(0,1,4)$-connected minor $N$ of $M$ such that, for all
$i\in\{4,5,\ldots,k\}$, the tangle $T_i(\cT)$ generates a tangle $T'_i(\cT)$ in $N$.
Moreover, $\breadth_M(T_i(\cT))=\breadth_N(T'_i(\cT))$.
\end{conjecture}

It is shown in Section~\ref{back} that we cannot do better than weakly 4-connected as an outcome.
This is because of the requirement of preserving breadth. Given the results of \cite{carm},
one could expect to sacrifice a constrained amount of breadth to arrive at an internally
4-connected minor. The following conjecture may not be difficult.

\begin{conjecture}
\label{get-internal}
Let $\cT$ be a tangle of order $k\geq 4$ and breadth $m$ in a matroid $M$. Then $M$ has an 
internally 
$4$-connected minor $N$ with a tangle $\cT'$ of order $k$ such that 
$\cT$ generates $\cT'$ in $N$ and such that the breadth of $\cT'$ is at least $m/2$.
\end{conjecture}

\subsection*{Acknowledgement} We thank Jim Geelen for pointing out the connection with 
$4$-connected sets. We also thank an anonymous referee for their careful reading of the manuscript.
This referee showed a mastery of coding theory in both detecting and often correcting a 
number of inexcusable errors.


\begin{thebibliography}{99} 
 
 \bibitem{bixby} R. Bixby, A simple theorem on $3$-connectivity, \emph{Linear Algebra Appl.} {\bf 45}
 (1982) 123--126.

\bibitem{carm} J. Carmesin and J. Kurkofka, Characterising 4-tangles through a connectivity property, 
 arXiv:2309.00902v1 (2023).

\bibitem{CX12} R. Chen and K. Xiang,  Decomposition of $3$-connected representable matroids, \emph{J. Combin. Theory Ser. B}  {\bf 102}  (2012),   647--670. 


\bibitem{D96} J. Dharmatilake,   A min-max theorem using matroid separations, \emph{Matroid theory} (Seattle, WA, 1995), 333--342, Contemp. Math., 197, Amer. Math. Soc., Providence, RI, 1996.

\bibitem{doov} G. Ding, B. Oporowski, J. Oxley, and
D. Vertigan, Unavoidable minors of large $3$-connected matroids,
{\em J. Combin.\ Theory Ser.\ B} {\bf 71} (1997), 244--293.


\bibitem{GZ15} J. Geelen  and S. H. M. van Zwam, Matroid $3$-connectivity and branch width, \emph{J. Combin. Theory Ser. B} {\bf 112} (2015), 104--123.

\bibitem{GGW09} J. Geelen, B. Gerards,   and G. Whittle,  Tangles, tree-decompositions and grids in matroids, \emph{J. Combin. Theory Ser. B}  {\bf 99} (2009),  657--667.
	
\bibitem{GGRW06} J. Geelen, B. Gerards, N. Robertson,   and G. Whittle, Obstructions to branch-decomposition of matroids,  \emph{J. Combin. Theory Ser. B} {\bf 96} (2006),   560--570.

\bibitem{H15} D. Hall, A characterization of tangle matroids, \emph{Ann. Combin.} {\bf 19} (2015), 125--130. 

\bibitem{L89} M. Lemos, On $3$-connected matroids, 
\emph{Discrete Math.} {\bf 73} (1989),   273--283.

\bibitem{O03} J.~Oxley, The structure of a $3$-connected matroid with a $3$-separating set of essential elements, \emph{Discrete Math.} {\bf 265} (2003), 173--187.


\bibitem{O11} J.~Oxley, \emph{Matroid Theory}, Second edition, Oxford University Press, New York, 2011. 

\bibitem{OSW04} J. Oxley, C. Semple, and G. Whittle,  The structure of the $3$-separations of $3$-connected matroids,  \emph{J. Combin. Theory Ser. B} {\bf 92}   (2004),  257--293.

 

\bibitem{RS91} N. Robertson   and P. D. Seymour,  Graph minors. X. Obstructions to tree-decomposition,  \emph{J. Combin. Theory Ser. B} {\bf 52} (1991),  153--190.



\bibitem{wT66} W.~T.~Tutte,  Connectivity in matroids, \emph{Canad. J. Math.} {\bf 18} (1966), 1301--1324.










\end{thebibliography}
\end{document}